\documentclass[12pt,reqno]{amsart}
\usepackage[margin=1in]{geometry}                
\usepackage{graphicx}
\usepackage{amssymb}
\usepackage{hyperref}
\usepackage{epstopdf}
\usepackage{color}
\usepackage{subcaption}
\usepackage{tikz}
\usepackage{url}
\usepackage[shortlabels]{enumitem}

\usetikzlibrary{calc}
\tikzstyle{vertex}=[circle, draw, inner sep=0pt, minimum size=6pt, fill=black]
\newcommand{\vertex}{\node[vertex]}

\theoremstyle{plain}
\newtheorem{thm}{Theorem}[section]
\newtheorem{lemma}[thm]{Lemma}
\newtheorem{prop}[thm]{Proposition}

\newtheorem{conj}[thm]{Conjecture}
\newtheorem*{thm*}{Theorem}
\newtheorem*{lemma*}{Lemma}
\newtheorem*{prop*}{Lemma}
\newtheorem*{cor*}{Corollary}
\newtheorem*{conj*}{Conjecture}

\theoremstyle{definition}
\newtheorem{defn}[thm]{Definition}
\newtheorem*{defn*}{Definition}
\newtheorem{ex}[thm]{Example}

\theoremstyle{remark}
\newtheorem{rmk}[thm]{Remark}

\newcommand{\bfx}{\mathbf{x}}

\newcommand{\rank}{\textnormal{rank}}

\DeclareMathOperator\trop{trop}
\newcommand{\cm}{{\rm CM}}

\newcommand{\initial}{{\rm in}}
\DeclareMathOperator\argmax{argmax}
\DeclareMathOperator{\berg}{\tilde{\mathcal{B}}}
\DeclareMathOperator{\cl}{cl}
\DeclareMathOperator{\contract}{/ \,}
\DeclareMathOperator{\aut}{Aut}
\DeclareMathOperator{\source}{so}
\DeclareMathOperator{\target}{ta}

\DeclareMathOperator{\aff}{Aff}

\title{Generic symmetry-forced infinitesimal rigidity: translations and rotations}
\author{Daniel Irving Bernstein}
\address{Department of Mathematics, Tulane University, New Orleans LA, USA}
\email{dibernst@mit.edu}\urladdr{https://dibernstein.github.io}

\begin{document}

\begin{abstract}
We characterize the combinatorial types of symmetric frameworks in the plane that are minimally generically symmetry-forced infinitesimally rigid when the symmetry group consists of rotations and translations. Along the way, we use tropical geometry to show how a construction of Edmonds that associates a matroid to a submodular function can be used to give a description of the algebraic matroid of a Hadamard product of two linear spaces in terms of the matroids of each linear space. This leads to new, short, proofs of Laman's theorem, and a theorem of Jord{\'a}n, Kaszanitzky, and Tanigawa, and Malestein and Theran characterizing the minimally generically symmetry-forced rigid graphs in the plane when the symmetry group contains only rotations.
\end{abstract}

\maketitle

\section{Introduction}

The vertex and edge sets of a graph $G$ will be denoted $V(G)$ and $E(G)$.
A \emph{$d$-dimensional (bar and joint) framework} is a pair $(G,p)$, consisting of a graph $G$ and a function $p:V(G)\rightarrow \mathbb{R}^d$.
We will often view such functions as points in $(\mathbb{R}^d)^{V(G)}$,
writing $p(u)$ as $p^{(u)}$ and $p(u)_i$ as $p^{(u)}_i$.
Intuitively, one should think about a framework as a physical construction of $G$, where an edge between vertices $u$ and $v$ is a rigid bar of length $\|p(u)-p(v)\|$, free to move around each of its incident vertices.
A \emph{motion} of a framework $(G,p)$ in $\mathbb{R}^d$ is a
continuous function $f: [0,1]\rightarrow (\mathbb{R}^d)^{V(G)}$ such that $f(0) = p$
and $\|f(t)^{(u)}-f(t)^{(v)}\|_2^2 = \|p^{(u)}-p^{(v)}\|_2^2$ for all $uv \in E(G)$ and $t \in [0,1]$.
Every direct Euclidean isometry of $\mathbb{R}^d$ gives rise to a motion of $(G,p)$ and such motions are called \emph{trivial}.
A framework is said to be \emph{rigid} if its only motions are trivial.
See Figure~\ref{fig:fourCycleNotRigid} for an example.

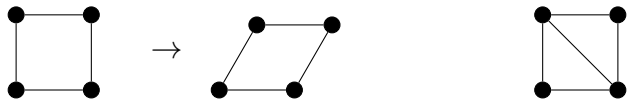
\begin{figure}
    \begin{tikzpicture}
        \vertex (1) at (0,0){};
        \vertex (2) at (1,0){};
        \vertex (3) at (1,1){};
        \vertex (4) at (0,1){};
        \path
            (1) edge (2) edge (4)
            (2) edge (3)
            (3) edge (4)
        ;
        \node at (2,1/2){$\rightarrow$};
        \vertex (1a) at (0+2.7,0){};
        \vertex (2a) at (1+2.7,0){};
        \vertex (3a) at ($ (3/2+2.7,{sqrt(3)/2}) $){};
        \vertex (4a) at ($ (1/2+2.7,{sqrt(3)/2}) $){};
        \path
            (1a) edge (2a) edge (4a)
            (2a) edge (3a)
            (3a) edge (4a)
        ;
        \vertex (1b) at (0+7,0){};
        \vertex (2b) at (1+7,0){};
        \vertex (3b) at (1+7,1){};
        \vertex (4b) at (0+7,1){};
        \path
            (1b) edge (2b) edge (4b)
            (2b) edge (3b) edge (4b)
            (3b) edge (4b)
        ;        
    \end{tikzpicture}
    \caption{On the left, we see a nontrivial motion of a flexible framework in the plane. Before applying the flex, the framework is symmetric with respect to an order-four rotation about the center. The flex destroys this symmetry. The framework on the right is rigid.}\label{fig:fourCycleNotRigid}
\end{figure}

Asimow and Roth \cite{asimow1978rigidity} showed that for each fixed finite graph $G$,
rigidity in $\mathbb{R}^d$ is a generic property in the sense that either almost all $d$-dimensional frameworks
on $G$ are rigid, or almost all $d$-dimensional frameworks on $G$ are flexible (i.e.~not rigid).
Thus, for each $d \ge 1$, it becomes meaningful ask for a characterization of which graphs are \emph{generically rigid} in $\mathbb{R}^d$, i.e.~whether or not almost all $d$-dimensional frameworks on $G$ are rigid.

For $d = 1$, one can see intuitively that a graph is generically rigid if and only if it is connected.
Generically rigid graphs in $\mathbb{R}^2$ were characterized by
Hilda Pollaczek-Geiringer in 1927~\cite{pollaczek1927gliederung}.
Her work was evidently forgotten until recently, since her characterization often bears the name ``Laman's theorem'' 
due to its rediscovery in 1970 by Gerard Laman~\cite{laman1970graphs}.
Her theorem, stated below as Theorem~\ref{thm:lamansTheorem}, has been used to develop a polynomial-time graph pebbling algorithm for determining whether a given graph is generically rigid in the plane~\cite{jacobs1997algorithm}.
Characterizing generic rigidity in three or more dimensions is open.

\begin{thm}[\cite{pollaczek1927gliederung}]\label{thm:lamansTheorem}
    A graph $G$ is \emph{minimally} generically rigid in the plane if and only if for every subgraph $H$ of $G$,
    \[
        |E(H)| \le 2 |V(H)| - 3
    \]
    with equality when $H = G$. A graph is generically rigid in the plane if and only if it contains a minimally generically rigid subgraph.
\end{thm}

Questions about rigidity of frameworks appear in diverse applications, including
sensor network localization~\cite{yang2012detecting}, biochemistry~\cite{whiteley2005counting},
civil engineering~\cite{kanno2001group,tarnai1980simultaneous},
and crystallography \cite{wegner2007rigid,sartbaeva2006flexibility,fowler2002symmetry,guest2003determinacy}.
Frameworks appearing in the latter two applications, especially in crystallography, often have symmetry constraints.
We now formally define symmetries of frameworks.

\begin{defn}\label{defn:symmetry}
    If $(G,p)$ is a $d$-dimensional framework and $T:\mathbb{R}^d\rightarrow \mathbb{R}^d$ is a Euclidean isometry,
    then we say that $T$ is a \emph{symmetry of $(G,p)$} if
    \begin{enumerate}
        \item $T$ induces a bijection $\phi:V(G)\rightarrow V(G)$ on the vertices of $G$, (i.e.~if for each $v \in V(G)$, there exists $\phi(v) \in V(G)$ such that $T(p(v)) = p(\phi(v))$ and $\phi:V(G)\rightarrow V(G)$ is a bijection),
        \item $uv$ is an edge of $G$ if and only if $\phi(u)\phi(v)$ is as well, and
        \item $\phi(v) \neq v$ for all $v \in V(G)$.
    \end{enumerate}
    The symmetries of a framework form a group. If $\mathcal{S}$ is a subgroup of isometries of $d$-dimensional Euclidean space such that each $T \in \mathcal{S}$ is a symmetry of $(G,p)$, then we say that $(G,p)$ is \emph{$\mathcal{S}$-symmetric}.
\end{defn}

\begin{rmk}
    Some authors would consider any $T$ satisfying the first two conditions in Definition~\ref{defn:symmetry} a symmetry, and say that $T$ is a \emph{free} symmetry if it satisfies all three.
\end{rmk}

If a symmetric framework is flexible, then its motions may break the symmetry. Consider for example the framework on the left in Figure~\ref{fig:fourCycleNotRigid} - it has symmetry with respect to the order-four rotation about its center, but the indicated motion breaks that symmetry. Visit the URL in the caption of Figure~\ref{figure:wallpaperSymmetry} for an animation of a motion of a symmetric framework that preserves symmetry.
If all motions of a given $\mathcal{S}$-symmetric framework $(G,p)$ are either trivial, or break the symmetry at some point during the motion, then we say that $(G,p)$ is \emph{$\mathcal{S}$-symmetry forced rigid}.

The most compact way to represent an $\mathcal{S}$-symmetric framework is using an \emph{$\mathcal{S}$-gain graph}, which is a directed multigraph whose arcs are labeled by elements of $\mathcal{S}$. More formally, if $G$ is a directed multigraph, we let $V(G)$ and $A(G)$ denote its vertex and arc set, and an $\mathcal{S}$-gain graph is a pair $(G,\phi)$ where $\phi: A(G)\rightarrow \mathcal{S}$ is a function. An $\mathcal{S}$-symmetric framework in $d$ dimensions is then encoded as $(G,\phi,p)$ where $p:V(G)\rightarrow \mathbb{R}^d$. This will be fleshed out in more detail later, but the rough idea is that $V(G)$ now corresponds to vertex \emph{orbits} of some framework with $\mathcal{S}$ symmetry, and the function $p$ specifies the location of a single vertex in each orbit. The locations of the remaining vertices are then determined via the $\mathcal{S}$-action. The arcs and their labels specify edge orbits in a similar way. See Figure~\ref{fig:representAsGainGraph} for an example.

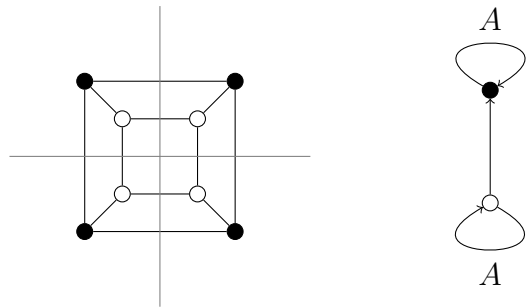
\begin{figure}
    \begin{tikzpicture}
        \vertex[fill=white] (a1) at (1/2,1/2){};
        \vertex[fill=white] (b1) at (-1/2,1/2){};
        \vertex[fill=white] (c1) at (-1/2,-1/2){};
        \vertex[fill=white] (d1) at (1/2,-1/2){};
        \vertex (a2) at (1,1){};
        \vertex (b2) at (-1,1){};
        \vertex (c2) at (-1,-1){};
        \vertex (d2) at (1,-1){};
        \path
            (a1) edge (b1) edge (d1) edge (a2)
            (b1) edge (c1) edge (b2)
            (c1) edge (d1) edge (c2)
            (d1) edge (d2)
            (a2) edge (b2) edge (d2)
            (b2) edge (c2)
            (c2) edge (d2)
        ;
        \draw[gray] (0,-2) -- (0,2);
        \draw[gray] (-2,0) -- (2,0);
    \end{tikzpicture}
    \qquad
    \begin{tikzpicture}
        \node at (0,-2){};
        \node at (0,2){};
        \vertex[fill=white] (a) at (0,-.75){};
        \vertex (b) at (0,.75){};
        \path[->] (a) edge (b){};
        \draw[->] (a) to [out=330,in=210,looseness=20] (a);
        \draw[->] (b) to [out=150,in=30,looseness=20] (b);
        \node at (0,1.7){$A$};
        \node at (0,-1.7){$A$};
    \end{tikzpicture}
    \caption{The framework on the left is $\mathcal{S}$-symmetric where $\mathcal{S}$ is generated by a ninety-degree rotation $A$.
    On the right is a gain graph expressing such a framework
    (the identity gain label is suppressed).
    }\label{fig:representAsGainGraph}
\end{figure}

We will only consider $\mathcal{S}$-gain graphs with finitely many vertices, but even in this case,
the corresponding symmetric frameworks will have infinitely many vertices when $\mathcal{S}$ is infinite (e.g.~Figure~\ref{figure:wallpaperSymmetry}).
Care must therefore be taken when adapting the usual linear-algebraic techniques of rigidity theory
since some of the relevant vector spaces are infinite-dimensional (see \cite{owen2011infinite,power2014crystal} for a functional-analytic approach to this problem).
In particular, it is unclear whether rigidity is a generic property of a gain graph
for infinite subgroups of Euclidean groups that contain arbitrarily small nonzero translations.
Fortunately however, \emph{infinitesimal symmetry-forced rigidity}, a stronger notion of symmetry-forced rigidity
that is equivalent to symmetry-forced rigidity for many groups if one invokes a certain genericity assumption~\cite{borcea2010periodic,schulze2011orbit,malestein2014frameworks},
is a generic property of $\mathcal{S}$-gain graphs.

The main result of this paper, Theorem~\ref{thm:mainTheoremTranslationsAndRotations}, is a combinatorial characterization of the $\mathcal{S}$-gain graphs that are generically minimally infinitesimally rigid in
$\mathbb{R}^2$ when $\mathcal{S}$ is a subgroup of the
\emph{direct} Euclidean isometries of $\mathbb{R}^2$.
This generalizes Theorem~\ref{thm:lamansTheorem} and several analogous results for symmetry-forced rigidity that have appeared previously.
Ross~\cite{ross2015inductive} characterizes the $\mathbb{Z}^2$-gain graphs
that are generically rigid in $\mathbb{R}^2$.
Malestein and Theran~\cite{malestein2013generic} study a generalization of $\mathbb{Z}^2$-forced rigidity
that allows the lattice to vary, characterizing generic rigidity in this situation. They also characterize symmetry-forced generic rigidity for orientation-preserving wallpaper groups with flexible lattices in~\cite{malestein2014frameworks}, and rotation groups~\cite{malestein2015frameworks}.
Jord{\'a}n, Kaszanitzky, and Tanigawa~\cite{jordan2016gain} characterize the $\mathcal{S}$-gain graphs that are
generically rigid when $\mathcal{S}$ is a rotation group, or a dihedral group whose rotation subgroup has odd order.
The new groups covered by our main theorem are the frieze group generated by a 180-degree rotation and a translation, and all non-discrete groups of direct Euclidean isometries of $\mathbb{R}^2$.

We now sketch the big ideas behind the proof of our main theorem.
For each group $\mathcal{S}$ of Euclidean isometries of $\mathbb{R}^d$,
we define the \emph{symmetric Cayley-Menger variety} $\cm_n^\mathcal{S}$.
When $\mathcal{S}$ is the trivial group, $\cm_n^\mathcal{S}$ is the usual Cayley-Menger variety, which is parameterized by the pairwise distances among $n$ points in $\mathbb{R}^d$.
The symmetric Cayley-Menger variety lives in a possibly infinite dimensional space whose coordinates are indexed by the arcs of the \emph{complete $\mathcal{S}$-gain graph on $n$ vertices}, which we define to be the gain graph on $n$ vertices with $|\mathcal{S}|$ arcs between each pair of vertices, each labeled by a distinct element of $\mathcal{S}$, and $|\mathcal{S}|-1$ loops at each vertex, each labeled by a distinct non-identity element of $\mathcal{S}$ (see Figure~\ref{fig:completeGainGraph}). The generically infinitesimally rigid $\mathcal{S}$-gain graphs will then be the gain graphs $(G,\phi)$ such that the coordinate projection of $\cm_n^\mathcal{S}$ onto the arcs of $(G,\phi)$ preserves dimension. By definition, these are the spanning sets in the algebraic matroid of $\cm_n^\mathcal{S}$. Our main tools for demystifying this matroid are tropical geometry (particularly Bergman fans), a construction of Edmonds that creates matroids from submodular functions, and matroid lifts.

\begin{figure}
    \begin{tikzpicture}
        \node at (-2.5,0){$K_2(\mathbb{Z}_2)=$};
        \vertex (1) at (0,0){};
        \vertex (2) at (1,0){};
        \draw[->] (1) to [out=135,in=225,looseness=20] (1);
        \node at (-1,0){1};
        \draw[->] (2) to [out=45,in=-45,looseness=20] (2);
        \node at (2,0){1};
        \draw[->] (1) to [out=45,in=135] (2);
        \node at (0.5,0.5){0};
        \draw[->] (1) to [out=-45,in=225] (2);
        \node at (0.5,-0.5){1};
    \end{tikzpicture}
    \caption{The complete $\mathbb{Z}_2$-gain graph on two vertices.}\label{fig:completeGainGraph}
\end{figure}
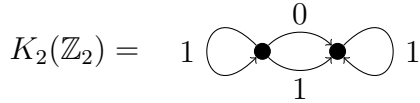

We now outline the remainder of this paper.
Section~\ref{section:preliminaries} provides some initial background in rigidity theory, matroid theory, and algebraic geometry.
Section~\ref{section:mainResult} contains Theorem~\ref{thm:mainResult}, the engine behind everything else in this paper, and the main result of earlier arXiv versions.
It uses a construction of Edmonds to describe the algebraic matroid of a Hadamard product of linear spaces.
Its proof uses tropical geometry.
We show how Pollaczek-Geiringer's characterization of minimal generic rigidity in the plane (Theorem~\ref{thm:lamansTheorem}) is an easy consequence of Theorem~\ref{thm:mainResult}.
Section~\ref{section:symmetry} contains the necessary background on gain graphs and symmetry-forced rigidity.
Theorem~\ref{thm:cyclicSymmetryMatroid} is a result concurrently proven by Jord{\'a}n, Kaszanitzky, and Tanigawa in~\cite{jordan2016gain} and Malestein and Theran in~\cite{malestein2015frameworks}.
It characterizes the $\mathcal{S}$-gain graphs
that are generically minimally rigid in $\mathbb{R}^2$ when $\mathcal{S}$ is a rotation group.
We show how it follows from Theorem~\ref{thm:mainResult}
in the same way that Pollaczek-Geiringer's theorem does.
Section~\ref{section:symmetryGroupsWithTranslation} begins with the background on matroid lifts required for Theorem~\ref{thm:mainAffine},
a generalization of Theorem~\ref{thm:mainResult} that describes the algebraic matroid of a Hadamard product of \emph{affine} spaces.
We use this to prove the paper's main result,
Theorem~\ref{thm:mainTheoremTranslationsAndRotations}, which characterizes generic minimal infinitesimal rigidity 
of $\mathcal{S}$-gain graphs in $\mathbb{R}^2$ when $\mathcal{S}$ consists of translations and rotations.

\section*{Acknowledgments}
Thanks to Bill Jackson for asking the question that spawned this project
and for a particularly helpful conversation.
This work began at the 2020 Heilbronn focused research group on discrete structures
at Lancaster University in the UK.
The author was supported by a US NSF Mathematical Sciences Postdoctoral Research Fellowship (DMS-1802902).

\section{Infinitesimal rigidity, matroids, and tropical geometry}\label{section:preliminaries}
This section lays out some initial preliminaries.
We assume that the reader is familiar with the basics of matroid theory.
Any undefined matroid-theoretic symbols and terms can be found in \cite[Chapters 1-3]{oxley2006matroid}.

\subsection{Infinitesimal rigidity}

An \emph{infinitesimal motion} of a $d$-dimensional framework $(G,p)$ is a vector $g \in (\mathbb{R}^d)^{V(G)}$
satisfying $\langle g^{(u)}-g^{(v)}, p^{(u)}-p^{(v)}\rangle = 0$ for all $u,v \in E(G)$.
In other words, an infinitesimal motion assigns a direction to each vertex such that when the vertices
move an \emph{infinitesimal} amount in the assigned directions, the edges neither stretch nor contract.
If $f$ is a motion of $(G,p)$, then $f'(0)$ is an infinitesimal motion of $(G,p)$.
When $g = f'(0)$ for a trivial motion $f$, then $g$ is said to be a \emph{trivial infinitesimal motion}.

The set of infinitesimal motions of a framework $(G,p)$ is a linear subspace of $(\mathbb{R}^d)^{V(G)}$.
When the only infinitesimal motions of a framework are trivial, the framework is said to be \emph{infinitesimally rigid}.
The dimension of the set of direct Euclidean isometries of $\mathbb{R}^d$,
and therefore the dimension of the linear space of trivial infinitesimal motions, is $\binom{d+1}{2}$.
Thus a framework is infinitesimally rigid if and only if its linear space of infinitesimal motions
has dimension $\binom{d+1}{2}$.
If a framework is infinitesimally rigid, then it is also rigid (differentiate any nontrivial motion to obtain an infinitesimal one).
An example showing failure of the converse statement appears in Figure~\ref{fig:rigidButNotInf}.

\begin{figure}
    \begin{tikzpicture}
        \vertex (a) at (0,0){};
        \vertex (b) at (2,0){};
        \vertex[fill=white] (c) at (1,0){};
        \vertex (d) at (1,1.2){};
        \vertex (e) at (1,2){};
        \path
            (a) edge (c) edge (d) edge (e)
            (b)  edge (c) edge (d) edge (e)
            (d) edge (e)
        ;
    \end{tikzpicture} \qquad\qquad
    \begin{tikzpicture}
        \vertex (a) at (0,0){};
        \vertex (b) at (2,0){};
        \vertex[fill=white] (c) at (1,0.5){};
        \vertex (d) at (1,1.2){};
        \vertex (e) at (1,2){};
        \path
            (a) edge (c) edge (d) edge (e)
            (b)  edge (c) edge (d) edge (e)
            (d) edge (e)
        ;
    \end{tikzpicture}
    \caption{The framework on the left is rigid, but not infinitesimally rigid.
    To see this, consider a path $f:[0,1]\rightarrow (\mathbb{R}^2)^5$ that slides the unshaded vertex up,
    while keeping the others fixed.
    This is \emph{not} a motion, since it stretches the edges that are incident to the unshaded vertex,
    but $f'(0)$ is a nontrivial infinitesimal motion.
    The framework on the right is both rigid and infinitesimally rigid.
    }\label{fig:rigidButNotInf}
\end{figure}
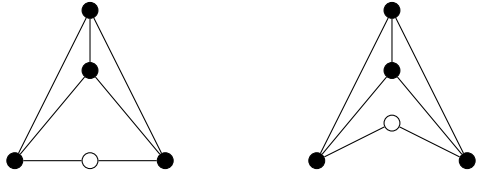

Just as in the case of ordinary rigidity, results of Asimow and Roth~\cite{asimow1979rigidity} show that for each fixed finite graph $G$, infinitesimal rigidity is a generic property in the sense that either $(G,p)$ is rigid for almost all $p:V(G)\rightarrow \mathbb{R}^d$, or $(G,p)$ is flexible for almost all $p:V(G)\rightarrow \mathbb{R}^d$.
Thus one says that a graph is \emph{generically infinitesimally rigid in $d$ dimensions} if $(G,p)$ is infinitesimally rigid for almost all $p:V(G)\rightarrow \mathbb{R}^d$.

\begin{thm}[\cite{asimow1979rigidity}]\label{thm:genericallyRigidIffInf}
    A graph is generically rigid in $d$ dimensions if and only if it is generically infinitesimally rigid in $d$ dimensions.
\end{thm}

\subsection{Algebraic matroids}
Given a (possibly infinite) set $E$, we denote its power set by $2^E$
and the complex vector space whose coordinates are in bijection with $E$ by $\mathbb{C}^E$.
A \emph{variety} in $\mathbb{C}^E$ is a set of the form
\[
    \{x \in \mathbb{C}^E: f(x) = 0 \textnormal{ for all } f \in I\}
\]
where $I$ is a set of polynomial functions on $\mathbb{C}^E$.
Polynomials, by definition, are only allowed to have finitely many terms, even when $E$ is infinite.
We let $\overline{X}$ denote the \emph{Zariski closure} of $X$,
i.e.~the smallest variety in $\mathbb{C}^E$ containing $X$.
Each subset $S \subseteq E$ canonically defines a coordinate projection $\pi_S: \mathbb{C}^E\rightarrow \mathbb{C}^S$.
The main class of matroids we will consider comes from parameterized varieties in the following way.

\begin{defn}\label{defn:algebraicMatroid}
    Let $E,F$ be sets with $F$ finite, let $f:\mathbb{C}^F\rightarrow \mathbb{C}^E$ be a polynomial map
    and define $V := \overline{f(\mathbb{C}^F)}$.
    The \emph{algebraic matroid of $V$}, denoted $\mathcal{M}(V)$,
    has as independent sets the subsets $I \subseteq E$ such that $\dim(\pi_I(V)) = |I|$.
\end{defn}

The rank of a matroid as in Definition~\ref{defn:algebraicMatroid} is the dimension of $V$.
This dimension is bounded above by $|F|$ and thus finite.
The matroid of linear independence on the rows of the Jacobian of $f$, evaluated at a generic point of $\mathbb{C}^F$,
is isomorphic to $\mathcal{M}(V)$.
We now introduce the \emph{Cayley-Menger variety}, whose algebraic matroid is relevant to rigidity theory.

\begin{defn}\label{defn:cayleyMenger}
    Let $n\ge d$ be integers and let $E$ be the edge set of the complete graph on $n$ vertices.
    Define the map $D^d_n: (\mathbb{R}^d)^n\rightarrow \mathbb{R}^{E}$
    so that for $ij \in E$,
    \[
        D^d_n(x^{(1)},\dots,x^{(n)})_{ij} = \|x^{(i)}-x^{(j)}\|_2^2.
    \]
    The \emph{Cayley-Menger variety} $\cm_n^d$
    is the Zariski closure of $D^d_n((\mathbb{R}^d)^n)$.
\end{defn}

Connections between Cayley-Menger varieties and rigidity theory were first noted in \cite{borcea2004number,borcea2002point}.
The algebraic matroid of $\cm_n^d$ is often called the \emph{$d$-dimensional rigidity matroid}
because of the following well-known proposition.
We provide a proof for the sake of exposition.

\begin{prop}[Folklore]\label{prop:rigidIffSpanning}
    A graph $G$ is generically rigid in $d$-dimensions
    if and only if $E(G)$ is spanning in the algebraic matroid of $\cm_n^d$.
\end{prop}
\begin{proof}
    It is a straightforward computation to see that the set of infinitesimal motions of $(G,p)$ is the nullspace
    of the differential of $\pi_{E(G)}\circ D_n^d$ at $p$.
    Thus $(G,p)$ is infinitesimally rigid when the submatrix of the differential of $\pi_{E(G)}\circ D_n^d$
    has maximum possible rank as $(G,p)$ is allowed to vary.
    For fixed $G$ and variable $p$, the rank of this differential attains its maximum value $r$ as long as some $r\times r$ submatrix has nonvanishing determinant.
    The set of such $p$ is a Zariski-open subset of $(\mathbb{R}^d)^n$,
    i.e.~such $p$ are generic.
    The dimension of the image of a polynomial map is equal to the rank of its differential at a generic point,
    so $E(G)$ is spanning in $\mathcal{M}(\cm_n^d)$ if and only if $(G,p)$
    is infinitesimally rigid whenever $p$ is generic.
    The proposition now follows from Theorem~\ref{thm:genericallyRigidIffInf}.
\end{proof}

We now state a few simple, but illuminating, consequences of Proposition~\ref{prop:rigidIffSpanning} and Theorem~\ref{thm:genericallyRigidIffInf}.
The dimension of $\cm_n^d$ is $dn-\binom{d+1}{2}$, which is the minimum number of edges in a graph that is generically rigid in $d$-dimensional space.
The minimally generically rigid graphs in $d$-dimensions are the bases of $\mathcal{M}(\cm_n^d)$
(and characterizing them for $d=3$ is perhaps the most important open problem in rigidity theory).
Finally, Pollaczek-Geiringer's characterization of generic rigidity in the plane (Theorem~\ref{thm:lamansTheorem}) is equivalent to the following.

\begin{thm}[\cite{pollaczek1927gliederung}]\label{thm:lamansMatroid}
    Let $G$ be a graph on vertex set $\{1,\dots,n\}$. Then $E(G)$ is independent in $\mathcal{M}(\cm_n^2)$ if and only if 
    for every subgraph $H$ of $G$ with $m$ vertices,
    \[
        |E(H)| \le 2m-3.
    \]
    If $G$ is independent in $\mathcal{M}(\cm_n^2)$, it is moreover a basis if and only if $|E(G)| = 2n-3$.
\end{thm}

\subsection{Matroids from submodular functions}
A function $f: 2^E\rightarrow \mathbb{Z}$ is the rank function of a matroid
if and only if it is submodular, nonnegative, and satisfies $f(S) \le f(S \cup \{e\}) \le f(S) + 1$
for all $S \subset E$ and $e \in E\setminus S$ \cite[Chapter 1]{oxley2006matroid}.
We now describe a construction of Edmonds
that derives a matroid from a set function satisfying weaker conditions.

\begin{defn}[\cite{edmonds1970matroids}]\label{defn:matroidFromSubmodular}
    Let $E$ be a possibly infinite set and assume $f: 2^{E} \rightarrow \mathbb{Z}$ satisfies
    \begin{enumerate}
    	\item\label{item:submodularity} $f(S \cup T) + f(S \cap T) \le f(S) + f(T)$
        for all $A,B \subseteq E$ (submodularity),
        \item\label{item:monotonicity} $f(S) \le f(T)$ whenever $S \subseteq T \subseteq E$ (monotonicity), and
        \item\label{item:finiteRank} there exists $N \in \mathbb{Z}$ such that $f(S) \le N$ for all $S \subseteq E$ (boundedness).
    \end{enumerate}
    Then $\mathcal{M}(f)$ denotes the matroid on ground set $E$
    whose independent sets are
    \[
    	\mathcal{I} := \{I \subseteq E : I = \emptyset \ {\rm or} \ 
        |I'| \le f(I') \textnormal{ for all } I' \subseteq I\}.
    \]
\end{defn}

See \cite[Chapter 11.1]{oxley2006matroid} for a proof that
Definition \ref{defn:matroidFromSubmodular} indeed defines a matroid when $E$ is finite.
The infinite case easily follows given boundedness.
Given a matroid $M$, we let $r_M$ denote its rank function.
Note that $M = \mathcal{M}(r_M)$.
We now provide a more interesting example illustrating Definition~\ref{defn:matroidFromSubmodular}.


\begin{ex}\label{ex:rigidityIsDTMU}
    Let $M$ the graphic matroid of the complete graph on $n$ vertices.
    Then $\mathcal{M}(2r_M - 1)$ is the algebraic matroid of $\cm_n^2$.
    This was proven in \cite{lovasz1982generic}
    and we will use our main result to give a new proof of this.
\end{ex}

Another matroid construction we will require is the \emph{matroid union}.
Theorem~\ref{thm:unionFromSubmodular} below tells us that the matroid union is
a special case of Definition~\ref{defn:matroidFromSubmodular}.
It was originally proven by Pym and Perfect in a more general context \cite{pym1970submodular}.

\begin{defn}
    Let $M_1,\dots,M_k$ be matroids on a common ground set $E$.
    The \emph{union} of $M_1,\dots,M_k$, denoted $M_1 \vee \dots \vee M_k$,
    is the matroid where $I \subseteq E$ is independent if and only if $I = I_1 \cup \dots \cup I_k$
    where $I_i$ is independent in $M_i$.
\end{defn}

\begin{thm}[{\cite[Chapter 8.3, Theorem 2]{welsh2010matroid}}]\label{thm:unionFromSubmodular}
    If $f,g: 2^E \rightarrow \mathbb{Z}$ are submodular, monotone, and nonnegative,
    then $\mathcal{M}(f + g) = \mathcal{M}(f) \vee \mathcal{M}(g)$.
    In particular, if $M_1,\dots,M_k$ are matroids on a common ground set $E$, then
    $M_1 \vee \dots \vee M_k = \mathcal{M}(r_{M_1} + \dots + r_{M_k})$.
\end{thm}

\subsection{Hadamard products of varieties}
Hadamard products of varieties were introduced in \cite{cueto2010geometry,cueto2010implicitization}
to study the algebraic geometry of restricted Boltzmann Machines.
Their theory was further developed in \cite{bocci2017hadamard,bocci2016hadamard}
with particular attention to linear spaces.
Hadamard products have since become of more fundamental interest in algebraic geometry
\cite{bocci2018hilbert,calussi2018hadamard,carlini2019hadamard,friedenberg2017minkowski}.

\begin{defn}
	The \emph{Hadamard product} of two points $x,y \in \mathbb{K}^E$,
	denoted $x \star y$, is the point $z \in \mathbb{K}^E$ such that $z_e = x_ey_e$ for all $e \in E$.
	The Hadamard product of varieties $U,V \subseteq \mathbb{K}^E$ is
    \[
    	U \star V := \overline{\{x\star y : x \in U, y \in V\}}.
    \]
\end{defn}

\begin{ex}\label{ex:cayleyMengerHadamard}
    Let $L \subseteq \mathbb{C}^{E(K_n)}$ be the linear space parameterized as $d_{uv} = t_u - t_v$.
    The algebraic matroid of $L$ is the graphic matroid of $K_n$.
    We will show that $\cm^2_n = L \star L$.
    This seemingly innocuous observation unlocks some big theorems from tropical geometry for use in planar rigidity.
    The usual way to parameterize $\cm_n^2$, as in Definition~\ref{defn:cayleyMenger}, is the following
    \[
        z_{ij} = (x_i-x_j)^2 + (y_i-y_j)^2.
    \]
    As noted in~\cite{capco2018number}, applying the following change of variables
    \[
    	x_i \mapsto \frac{t_i + s_i}{2}
        \qquad y_i \mapsto \frac{t_i - s_i}{2\sqrt{-1}}
    \]
    yields the following
    \begin{align*}
        z_{ij} &= \left(\frac{t_i + s_i-t_j-s_j}{2}\right)^2 + \left(\frac{t_i - s_i-t_j+s_j}{2\sqrt{-1}}\right)^2 \\
        &= \frac{1}{4}\left(t_i^2+s_i^2+t_j^2+s_j^2+2t_is_i+2t_js_j-2t_is_j-2t_js_i -2t_it_j-2s_is_j\right.
        \\&\qquad\qquad\left.-t_i^2-s_i^2-t_j^2-s_j^2 + 2t_it_j+2s_is_j-2t_is_j-2t_js_i+2t_is_i+2t_js_j\right)\\
        &=t_is_i+t_js_j-t_is_j-t_js_i\\
        &=(t_i-t_j)(s_i-s_j)
    \end{align*}
    and therefore $\cm^2_n = L \star L$.
    Theorem~\ref{thm:mainResult}, says that this is responsible for Example~\ref{ex:rigidityIsDTMU}.
\end{ex}

The observations in Example~\ref{ex:cayleyMengerHadamard} are fundamental so we summarize them below.

\begin{prop}\label{prop:cmHadamardProduct}
    For $d = 2$, the Cayley-Menger variety $\cm^2_n$ is a Hadamard product of two linear spaces. In particular,
    if $L \subseteq \mathbb{C}^{E(K_n)}$ is parameterized as $d_{uv} = t_u - t_v$, then
    \[
        \cm^2_n = L\star L.
    \]
    Moreover, the algebraic matroid of $L$ is the graphic matroid of $K_n$.
\end{prop}

\subsection{Tropical geometry}
Given a variety $V \subseteq \mathbb{C}^E$, let $I(V)$ denote the ideal
of the polynomial ring $\mathbb{C}[x_e : e \in E]$ consisting of the polynomials $f$
such that $f(x) = 0$ for all $x \in V$.
The ideal of $\mathbb{C}[x_e : e \in E]$ generated by polynomials $f_1,\dots,f_r$
will be denoted $\langle f_1,\dots,f_r\rangle$.
Given $\alpha \in \mathbb{Z}^E$,
the shorthand $\bfx^\alpha$ denotes the monomial $\prod_{e \in E} x_e^{\alpha_e}$.

\begin{defn}\label{defn:tropicalizaton}
    Let $E$ be a finite set,
    let $f = \sum_{\alpha \in \mathcal{J}} c_\alpha \bfx^{\alpha} \in \mathbb{C}[x_e : e \in E]$,
    and let $\omega \in \mathbb{R}^E$.
    The \emph{initial form of $f$ with respect to $w$} is
    \[
        \initial_\omega f := \sum_{\alpha \in
        \argmax_{\beta \in \mathcal{J}}\omega\cdot\beta} c_\alpha \bfx^\alpha.
    \]
    The \emph{initial ideal} of an ideal $I \subseteq \mathbb{C}[x_e : e \in E]$
    with respect to $\omega \in\mathbb{R}^n$ is
    \[
        \initial_\omega I := \langle \initial_\omega f : f \in I \rangle.
    \]
    The \emph{tropicalization} of a variety $V \in \mathbb{C}^E$ is
    \[
        \trop(V) := \{\omega \in \mathbb{R}^n : \initial_\omega I(V) \textnormal{ contains no monomials}\}.
    \]
\end{defn}
The tropicalization of a complex variety is a polyhedral fan.
Perhaps the easiest way to see this is to note that it can be obtained from the Gr\"obner fan of $I(V)$
by removing the interiors of all cones corresponding to initial ideals that contain monomials.
When $V$ is a hypersurface, its tropicalization is the polyhedral fan consisting of the codimension-one
cones in the normal fan to the Newton polytope of the generator of the principal ideal $I(V)$.

It is tempting to conclude that when $I(V) = \langle f_1,\dots, f_r\rangle$,
\begin{equation}\label{eq:tropicalBasis}
    \trop(V) = \bigcap_{i = 1}^r \{\omega \in \mathbb{R}^n : \initial_\omega f_i \textnormal{ has no monomials}\}
\end{equation}
but this is in general false.
A generating set $f_1,\dots,f_r$ of $I(V)$ that does satisfy \eqref{eq:tropicalBasis}
is called a \emph{tropical basis}.
It was shown in \cite{bogart2007computing} that a tropical basis exists for every variety $V \subseteq\mathbb{C}^E$
(finite $E$).
Tropicalization preserves a lot of information about a variety.
In particular, it preserves the algebraic matroid structure.
The following lemma of Yu makes this precise.

\begin{lemma}[{\cite{yu2017algebraic}}]\label{lemma:tropPreservesAlgMatroid}
    Let $E$ be a finite set, let $V \subseteq \mathbb{C}^E$, and let $S \subseteq E$.
    Then $\dim(\pi_S(V)) = \dim(\pi_S(\trop(V)))$.
\end{lemma}

Lemma \ref{lemma:HadamardProductMinkowskiSum} below says that Hadamard products interact cleanly with tropicalization.
Recall that the Minkowski sum of sets $A,B\subseteq \mathbb{R}^E$
is the set $A + B := \{a + b : a \in A, b \in B\}$.

\begin{lemma}\label{lemma:HadamardProductMinkowskiSum}
    Let $E$ be a finite set and let
    $U,V \subseteq \mathbb{C}^E$ be irreducible varieties.
    Then $\trop(U\star V) = \trop(U) + \trop(V)$.
\end{lemma}
\begin{proof}
    This is an immediate consequence of Theorem 1.1 in \cite{sturmfels2007elimination}.
\end{proof}

Definition~\ref{defn:bergmanFan} below associates a polyhedral fan to each matroid $M$.
Proposition~\ref{prop:algMatroidOfLinearSpaceIsBergmanFanOfMatroid} tells us that if $M$
is the algebraic matroid of a linear space $L$,
then the polyhedral fan given in Definition~\ref{defn:bergmanFan} is the tropicalization of $L$.

\begin{defn}\label{defn:bergmanFan}
    Let $M$ be a matroid on ground set $E$.
    A \emph{flag of flats} of $M$ is a set of nested nonempty flats of $M$.
    Given $e \in M$ and a flag of flats,
    let $\mathcal{F}^e$ denote the minimal element of $\mathcal{F}$ containing $e$.
    For each flag of flats $\mathcal{F}$ of $M$, define
    \[
        K_\mathcal{F} := \{\omega \in \mathbb{R}^E: \omega_e \le \omega_f \textnormal{ when }
        \mathcal{F}^e \subsetneq \mathcal{F}^f \textnormal{ and }
        \omega_e = \omega_f \textnormal{ when } \mathcal{F}^e = \mathcal{F}^f\}.
    \]
    The \emph{Bergman fan} of a loopless matroid $M$, denoted $\berg(M)$, is defined to be
    \[
        \berg(M):= \bigcup_{\mathcal{F}} K_\mathcal{F}
    \]
    where the union is over all flags of flats of $M$.
\end{defn}

\begin{prop}[{\cite[Ch. 9]{sturmfels2002solving}}]\label{prop:algMatroidOfLinearSpaceIsBergmanFanOfMatroid}
    Let $E$ be a finite set and let $L \subseteq \mathbb{C}^E$ be a linear space not contained in a coordinate hyperplane.
    Then $\trop(L) = \berg(\mathcal{M}(L))$.
\end{prop}

\section{The algebraic matroid of a Hadamard product of linear spaces}\label{section:mainResult}
We begin this section by stating Theorem~\ref{thm:mainResult}, its main result.
It gives a combinatorial description of the algebraic matroid of a Hadamard product of two linear spaces in terms of the algebraic matroids of each linear space.

\begin{thm}\label{thm:mainResult}
    Let $U,V \subseteq \mathbb{C}^E$ be finite-dimensional linear subspaces and define $M := \mathcal{M}(U)$ and $N:= \mathcal{M}(V)$.
    If neither $U$ nor $V$ is contained in a coordinate hyperplane, then
    \[
        \mathcal{M}(U\star V) = \mathcal{M}(r_M + r_N - 1).
    \]
\end{thm}

Before proving Theorem~\ref{thm:mainResult},
we use it to give a short proof of the matroid-theoretic statement of Pollaczek-Geiringer's characterization of generic rigidity in the plane.

\begin{proof}[Proof of Theorem~\ref{thm:lamansMatroid}]
    Let $M$ denote the graphic matroid of the complete graph $K_n$ on vertex set $\{1,\dots,n\}$.
    For each set $S\subseteq E(K_n)$, consider the graph whose vertex set consists of the vertices incident to an edge of $S$ and whose edge set is $S$, and let $v(S)$ and $c(S)$ denote the number of vertices and connected components of this graph.
    Recall that the rank function of $M$ is $r_M(S) = v(S)-c(S)$.
    By Proposition~\ref{prop:cmHadamardProduct},
    $\cm_n^2$ is the Hadamard product of two linear spaces whose matroid is $M$.
    Theorem~\ref{thm:mainResult} therefore implies that $\mathcal{M}(\cm_n^2) = \mathcal{M}(2r_M-1)$.
    If $I$ is independent in $\mathcal{M}(2r_M-1)$,
    then $|I'| \le 2r_M(I')-1 = 2v(I') - 2c(I')-1 \le 2v(I') - 3$ for all $I' \subseteq I$.
    If $I$ is dependent in $\mathcal{M}(2r_M-1)$,
    let $I' \subseteq I$ be such that $|I'| > 2r_M(I') - 1$.
    Then there exists some $I'' \subseteq I'$ such that the graph on edge set $I''$
    is connected and $|I''| > 2r_M(I'')-1 = 2v(I'') -3$.
\end{proof}

We now turn to proving Theorem~\ref{thm:mainResult}.
When $E$ is finite, Lemmas~\ref{lemma:HadamardProductMinkowskiSum} and~\ref{lemma:tropPreservesAlgMatroid}, along with Proposition~\ref{prop:algMatroidOfLinearSpaceIsBergmanFanOfMatroid}, tell us that the algebraic matroid of a Hadamard product of linear spaces is determined by the dimensions of the images of the coordinate projections of $\berg(M) + \berg(N)$, where $M$ and $N$ are the matroids of each linear space.
To this end, we introduce the following definition that, via Lemma~\ref{lemma:graphFromConeInMinkowskiSumOfBergmanFans} below, reduces these dimension computations to an auxiliary combinatorics problem,
which via Lemma~\ref{lemma:independentIffCycleFree}, is a bridge to independence in $\mathcal{M}(r_M + r_N-1)$.

\begin{defn}
    Let $M,N$ be matroids on a common ground set $E$.
    Let $\mathcal{F},\mathcal{G}$ be flags of flats for $M$ and $N$, respectively.
    For each subset $S \subseteq E$,
    let $H_{\mathcal{F},\mathcal{G}}^S$ be the bipartite graph on partite sets
    $\mathcal{F},\mathcal{G}$ with an edge for each $e \in S$
    connecting $\mathcal{F}^e$ to $\mathcal{G}^e$.
\end{defn}

When $M$ and $N$ have finite rank, $H_{\mathcal{F},\mathcal{G}}^S$ will have finitely many vertices,
though it could have infinitely many edges.

\begin{lemma}\label{lemma:graphFromConeInMinkowskiSumOfBergmanFans}
    Let $M$ and $N$ be matroids of finite rank on a common ground set $E$ and let $S \subseteq E$.
    Then $\dim(\pi_S(\berg(M) + \berg(N)))$ is equal to the maximum
    rank of the graphic matroid of $H_{\mathcal{F},\mathcal{G}}^S$
    as $\mathcal{F}$ and $\mathcal{G}$ respectively
    range over flags of flats of $M$ and $N$.
\end{lemma}
\begin{proof}
    First note that
    \[
    	\berg(M) + \berg(N) = \bigcup_{\mathcal{F},\mathcal{G}} K_\mathcal{F} + K_\mathcal{G}
    \]
    where the union is taken over all pairs $\mathcal{F},\mathcal{G}$ such that $\mathcal{F},\mathcal{G}$
    are respectively flags of flats of $M$ and $N$.
    We claim that the linear hull of $K_\mathcal{F} + K_\mathcal{G}$ is the column span of the incidence matrix of
    $H_{\mathcal{F},\mathcal{G}}^E$.
    Indeed, the linear hull of $K_{\mathcal{F}}$ (resp.~$K_\mathcal{G}$) is the column span of the matrix
    $M_\mathcal{F}$ (resp. $M_\mathcal{G}$)
    whose columns are the characteristic vectors of the flats in $\mathcal{F}$ (resp.~$\mathcal{G}$).
    We can then apply an invertible sequence of column operations to $M_\mathcal{F}$ (resp. $M_\mathcal{G}$) to obtain the matrix $N_\mathcal{F}$ (resp.~$N_\mathcal{G}$) whose columns are in bijection with the flats in $\mathcal{F}$ (resp. $\mathcal{G}$), where the column corresponding to $F \in \mathcal{F}$ (resp.~$G\in \mathcal{G}$) is the characteristic vector of the set containing the elements in $F$ (resp.~$G$) that lie in no proper sub-flat of $F$ (resp.~$G$) in $\mathcal{F}$ (resp.~$\mathcal{G}$).
    The incidence matrix of $H_{\mathcal{F},\mathcal{G}}^E$ is $(N_\mathcal{F} | N_\mathcal{G})$, thus proving the claim.

    It follows that the projection of the linear hull of $K_\mathcal{F} + K_\mathcal{G}$ onto
    the coordinates indexed by $S \subseteq E$ is the column space of the incidence matrix of $H_{\mathcal{F},\mathcal{G}}^S$.
    The dimension of this linear space is the rank of the graphic matroid of $H_{\mathcal{F},\mathcal{G}}^S$.
\end{proof}

The closure operator of a matroid $M$ will be denoted by $\cl_M$.
Given a subset $I$ of the ground set of $M$, $M|I$ denotes the restriction of $M$ to $I$.

\begin{lemma}\label{lemma:independentIffCycleFree}
    Let $M$ and $N$ be matroids of finite rank on a common ground set $E$ and let $S \subseteq E$.
    Then $S$ is independent in $\mathcal{M}(r_M + r_N - 1)$
    if and only if there exist flags of flats $\mathcal{F},\mathcal{G}$ of $M$ and $N$
    such that $H_{\mathcal{F},\mathcal{G}}^S$ is cycle-free.
\end{lemma}
\begin{proof}
    First assume there exist $\mathcal{F},\mathcal{G}$ such that $H_{\mathcal{F},\mathcal{G}}^S$ is cycle-free.
    Since $H_{\mathcal{F},\mathcal{G}}^S$ has finitely many vertices, $S$ is finite.
    For $I \subseteq E$,
    define $\mathcal{F}|I$ (respectively $\mathcal{G}|I$) to be the flags of flats of $M|I$ ($N|I$)
    obtained by restricting each $F \in \mathcal{F}$ ($G \in \mathcal{G}$) to $I$.
    The following argument shows that $S$ is independent in $\mathcal{M}(r_M + r_N - 1)$.
    Let $\eta_I: \mathcal{F}|I \cup \mathcal{G}|I \rightarrow \mathcal{F}\cup \mathcal{G}$
    be the map that sends each $F \in \mathcal{F}|I$ to the minimal
    element of $\mathcal{F}$ containing $F$,
    and each $G \in \mathcal{G}|I$ to the minimal element of $\mathcal{G}$ containing $G$.
    Then $\eta_I$ defines a graph homomorphism from $H_{\mathcal{F}|I,\mathcal{G}|I}^I$
    to $H_{\mathcal{F},\mathcal{G}}^S$
    that is an injection on edge sets.
    If $I \subseteq S$ has $r_M(I) + r_N(I)$ or more elements,
    then $H_{\mathcal{F}|I,\mathcal{G}|I}$ has a cycle.
    Via $\eta_I$, this gives a cycle in $H_{\mathcal{F},\mathcal{G}}$.

    Assume that $I$ is independent in $\mathcal{M}(r_M + r_N - 1)$.
    Then $I$ is finite.
    We construct flags of flats $\mathcal{F},\mathcal{G}$ of $M$ and $N$
    such that $H_{\mathcal{F},\mathcal{G}}^I$ is cycle-free.
    Since $I$ is also independent in $\mathcal{M}(r_M + r_N)$,
    Theorem~\ref{thm:unionFromSubmodular} implies $I = I_1 \cup I_2$
    for independent sets $I_1,I_2$ of $M,N$.
    By the augmentation axiom for independent sets of a matroid,
    we may assume that $I_1$ and $I_2$ span $I$ in $M$ and $N$ respectively.
    Given an ordering of the of the elements of $I_1$ and $I_2$, i.e.
    \[
        I_1 = \{e_1^{(1)},\dots,e_{r_M}^{(1)}\} \quad {\rm and} \quad I_2 = \{e_1^{(2)},\dots,e_{r_N}^{(2)}\}
    \]
    define $F_i := \cl_M(\{e_1^{(1)},\dots,e_i^{(1)}\})$ and
    $G_i := \cl_N(\{e_1^{(2)},\dots,e_i^{(2)}\})$
    and the flags of flats $\mathcal{F} := \{F_1,F_2,\dots,F_{r_M(I)}\}$
    and $\mathcal{G} := \{G_1,\dots,G_{r_N(I)}\}$.

    We now describe an iterative procedure to order $I_1$ and $I_2$
    so that the resulting $H_{\mathcal{F},\mathcal{G}}^I$ is cycle-free.
    Since $I$ is independent in $\mathcal{M}(r_M + r_N - 1)$, our spanning assumptions on $I_1$ 
    and $I_2$ imply $|I_1 \cap I_2| \ge 1$.
    We begin at the top, setting $e_{r_M(I)}^{(1)} = e_{r_N(I)}^{(2)}$ to be an element of $I_1 \cap I_2$.
    Now, assume $e_{r_M(I)}^{(1)},\dots,e_{r_M(I)-k_M}^{(1)}$
    and $e_{r_N(I)}^{(2)},\dots,e_{r_N(I)-k_N}^{(2)}$ have been set for some $k_M,k_N$ and define
    $I_1' = I_1\setminus \{e_{r_M(I)}^{(1)},\dots,e_{r_M(I)-k_M}^{(1)}\}$ and
    $I_2' = I_2\setminus \{e_{r_N(I)}^{(2)},\dots,e_{r_N(I)-k_N}^{(2)}\}$ to be the unindexed elements.
    For $0 \le i \le k_M$, define $S_i^{(1)}$ by
    \begin{align*}
        S_i^{(1)}:=\cl_M(\{e_{r_M(I)-i}^{(1)},\dots,&e_{r_M(I)-k_M}^{(1)}\} \cup I_1')\setminus\\
        &\left[\{e_{r_M(I)-i}^{(1)}\} \cup\cl_M(\{e_{r_M(I)-i-1}^{(1)},\dots,e_{r_M(I)-k_M}^{(1)}\} \cup I_1')\right].
    \end{align*}
    Note that $S_i^{(1)} \subseteq I_2\setminus I_1$.
    We will take as an inductive hypothesis that there exists some $0 \le i_0 \le k_M$ such that
    $S_i^{(1)} \subseteq I_2'$ whenever $i \ge i_0$, and $S_i^{(1)} \cap I_2' = \emptyset$ otherwise.
    For each $i_0 \le i \le k_M$, index $S_i^{(1)}$ using the largest available indices in any order.
    Switch the roles of $I_1$ and $I_2$ and of $M$ and $N$, define $S_j^{(2)}$ analogously for
    $0 \le j \le k_N$ and repeat, continuing until nothing new can be indexed in this way, i.e.~when
    \[
    	\bigcup_{i=0}^{k_M} S_i^{(1)} \cap I_2' = \bigcup_{i=0}^{k_N} S_i^{(2)} \cap I_1' = \emptyset.
    \]
    At this point, no matter how we index $I_1',I_2'$,
    the induced subgraph of $H_{\mathcal{F},\mathcal{G}}^I$
    whose vertex set is $F_{k_M},\dots,F_{r_M(I)},G_{k_N},\dots,G_{r_N(I)}$ is a tree.
    Moreover, since
    \begin{align*}
    	\bigcup_{i=0}^{k_M} S_i^{(1)} = I_2\setminus(\{e^{(2)}_{r_N(I)}\} \cup \cl_M(I_1')) \quad {\rm and} \quad 
    	\bigcup_{i=0}^{k_N} S_i^{(2)} = I_1\setminus(\{e^{(1)}_{r_M(I)}\} \cup \cl_M(I_2')),
    \end{align*}
    we also have $I_1' \subseteq \{e^{(1)}_{r_M(I)}\} \cup \cl_N(I_2')$ and  $I_2' \subseteq \{e^{(2)}_{r_N(I)}\}  \cup \cl_M(I_1')$.
    By definition of $I_1',I_2'$, we know $e^{(1)}_{r_M(I)} = e^{(2)}_{r_N(I)} \notin I_1' \cup I_2'$,
    i.e.~that $I_1' \subseteq \cl_N(I_2')$ and  $I_2' \subseteq \cl_M(I_1')$,
    and therefore $r_M(I_1' \cup I_2') + r_N(I_1' \cup I_2') = |I_1'| + |I_2'|$.
    Independence of $I_1' \cup I_2'$ in $\mathcal{M}(r_M + r_N - 1)$
    then implies that either $I_1' \cup I_2' = \emptyset$ or $I_1' \cap I_2' \neq \emptyset$.
    In the former case, we are done.
    In the latter, we repeat this procedure setting $I_1$ to $I_1'$ and $I_2$ to $I_2'$.
    Since $I_1\cap I_2 = \{e_{r_M(I)}\} \cup (I_1' \cap I_2')$, this will eventually terminate
    with $H_{\mathcal{F},\mathcal{G}}^I$ being
    a forest with $|I_1 \cap I_2|$ connected components.
\end{proof}

We now have the following matroid-theoretic generalization of Theorem~\ref{thm:mainResult}.

\begin{thm}\label{thm:mainResultMatroidVersion}
    Let $M$ and $N$ be loopless matroids of finite rank on a common ground set $E$.
    Then $I \subseteq E$ is independent in $\mathcal{M}(r_M + r_N - 1)$ if and only
    if $\dim(\pi_I(\berg(M) + \berg(N))) = |I|$.
\end{thm}
\begin{proof}
    This is an immediate consequence of Lemmas~Lemmas~\ref{lemma:graphFromConeInMinkowskiSumOfBergmanFans} and~\ref{lemma:independentIffCycleFree}.
\end{proof}

We are now ready to prove Theorem~\ref{thm:mainResult}.
\begin{proof}[Proof of Theorem~\ref{thm:mainResult}]
    First assume $E$ is finite.
    Since neither $U$ nor $V$ is contained in a coordinate hyperplane,
    $M$ and $N$ are loopless.
    Proposition~\ref{prop:algMatroidOfLinearSpaceIsBergmanFanOfMatroid} and Lemmas~\ref{lemma:tropPreservesAlgMatroid} and \ref{lemma:HadamardProductMinkowskiSum}
    imply that the desired result is a special case of Theorem~\ref{thm:mainResultMatroidVersion}.
    Now assume $E$ is infinite. Then $I\subseteq E$ is independent in $\mathcal{M}(U\star V)$
    if and only if $I$ is independent in $\mathcal{M}(U\star V)|I$.
    Since $\mathcal{M}(U\star V)$ has finite rank and $\mathcal{M}(\pi_I(U\star V)) = \mathcal{M}(U\star V)|I$,
    the finite case implies that
    $I$ is independent if and only if $|I| \le r_M(I)+r_N(I)-1$.
\end{proof}

It would be interesting to see if Theorem~\ref{thm:mainResult} could be generalized to allow Hadamard products
of more than two linear spaces.
To this end, we make the following conjecture.

\begin{conj}\label{conj:theOnlyOne}
    If $L_1,\dots,L_d \subseteq \mathbb{R}^E$ are finite-dimensional linear spaces with algebraic matroids $M_1,\dots,M_d$, then
    $\mathcal{M}(L_1\star\dots\star L_d) = \mathcal{M}(r_{M_1} + \dots + r_{M_d} - d + 1)$.
\end{conj}

\section{Gain graphs and symmetry-forced rigidity}\label{section:symmetry}
Gain graphs describe ordinary graphs modulo free actions of groups.
Their use in rigidity theory originates with
Ross~\cite{ross2011geometricDissertation} and Malestein and Theran~\cite{malestein2013generic}
(though in~\cite{malestein2013generic} they are called ``colored graphs'').
In this section, we provide the necessary background on gain graphs and symmetric frameworks
and use this to introduce an analogue of the Cayley-Menger variety for symmetry-forced rigidity.
With all this at our disposal, we use Theorem~\ref{thm:mainResult} to give a short proof
of~\cite[Theorem~1]{malestein2015frameworks}, which is also a special case of \cite[Theorem~6.3]{jordan2016gain}.

\subsection{Gain graphs, frame matroids, and symmetry}
For a more leisurely introduction to gain graphs, see~\cite[Chapter 2]{gross2001topological}.
As before, the vertex and edge sets of a graph $G$ will be denoted $V(G)$ and $E(G)$.
When $G$ is directed, $A(G)$ will denote its arc set and $E(G)$ will denote
the edge set of its underlying undirected graph.
Directed graphs will be allowed to have loops and parallel arcs
but undirected graphs we consider will all be simple.
The source and target of an arc $e$ will be denoted $\source(e)$ and $\target(e)$.
For $S \subseteq E(G)$, we let $V(S)$ denote the set of vertices incident to some edge in $S$,
and we let $C(S)$ denote the partition of $S$ by connected components of the graph $(V(S),S)$.

Given a group $\mathcal{S}$, an \emph{$\mathcal{S}$-gain graph}, or simply a \emph{gain graph},
is a pair $(G,\phi)$ consisting of a directed graph $G$
along with an arc-labeling $\phi: A(G)\rightarrow \mathcal{S}$.
The \emph{gain} of a walk in the undirected graph underlying $G$ is defined to be
the product of the arc labels, inverting whenever an arc is traversed backwards.
In this sense, the orientation of $G$ does not really matter since reversing the direction of an arc $e$
is the same as inverting $\phi(e)$.
A cycle in $(G,\phi)$ is \emph{balanced} if its gain is the identity element of $\mathcal{S}$.
Given a gain graph $(G,\phi)$, the \emph{frame matroid} $\mathcal{M}(G,\phi)$ of $(G,\phi)$
is the matroid whose ground set is $E(G)$
where $S \subseteq E(G)$ is independent whenever every connected component of the subgraph of $G$ on edge set $S$
has at most one cycle, which is not balanced.
For more on frame matroids see \cite{zaslavsky1989biased,zaslavsky1991biased}.

Let $G$ be an undirected simple graph.
An \emph{automorphism} of $G$ is a permutation $\pi:V(G)\rightarrow V(G)$
such that $\{u,v\} \in E(G)$ if and only if $\{\pi(u),\pi(v)\} \in E(G)$.
The set of automorphism of $G$ is a group under composition which we denote by $\aut(G)$.
Given a group $\mathcal{S}$, an \emph{$\mathcal{S}$-action on $G$} is a group homomorphism $\rho:\mathcal{S}\rightarrow \aut(G)$.
An $\mathcal{S}$-action $\rho$ is said to be \emph{free} if $\rho(g)(v) \neq v$ unless $g$ is the identity element of $\mathcal{S}$
(note that unlike in \cite{gross2001topological}, we do not require the same condition for the induced action on $E(G)$).
We will say that a graph $G$ is \emph{$\mathcal{S}$-symmetric} if there exists a free $\mathcal{S}$-action on $G$.

We will now describe how $\mathcal{S}$-gain graphs offer a compact way to describe $\mathcal{S}$-symmetric graphs.
Let $\mathcal{S}$ be a group and let $\rho$ be a free $\mathcal{S}$-action on a simple undirected graph $G$.
Note that $\rho$ extends to an $\mathcal{S}$-action on $E(G)$ by $g\{u,v\} = \{gu,gv\}$.
Let $V(G)\contract \rho$ and $E(G)\contract \rho$ respectively denote the orbits of the action of $\rho$ on $V(G)$ and $E(G)$.
The $\mathcal{S}$-orbit of a single vertex $v$ or edge $e$ will be denoted $\mathcal{S}v$ and $\mathcal{S}e$.
Let $U$ consist of exactly one representative from each vertex orbit in $V(G)\contract \rho$.
Since $\rho$ is free, each edge orbit in $E(G)\contract \rho$ can be written as $\{(gu,ghv): g \in \mathcal{S}\}$ for unique $h \in \mathcal{S}$ and $u,v \in U$.
Viewing each such edge-orbit as an arc from $\mathcal{S}u$ to $\mathcal{S}v$ labeled by the aforementioned $h$,
we obtain a gain graph on vertex set $V(G)\contract \rho$, which we we denote by $G\contract (\rho,U)$.

Conversely, we associate to a gain graph $(G,\phi)$ the \emph{covering graph},
which is the simple undirected graph $H$
with vertex set $\mathcal{S}\times V(G)$
and edge set
\[
    \{\{(g,u),(g\phi(e),v)\}: e = (u,v) \in A(G), g \in \mathcal{S}\}.
\]
Then $\mathcal{S}$ acts freely on the covering graph via the action $\rho$ defined by $\rho(h)(g,u) = (hg,u)$.
Moreover, setting $U = \{({\rm id},u): u \in V(G)\}$ gives $H\contract (\rho,U)=(G,\phi)$.

\subsection{Symmetry-forced rigidity and Cayley-Menger varieties}
Let $\mathcal{E}(d)$ denote the group of Euclidean isometries of $\mathbb{R}^d$
and let $\mathcal{S}$ be a subgroup of $\mathcal{E}(d)$.
An \emph{$\mathcal{S}$-symmetric framework} is a framework $(G,p)$ such that there exists an
$\mathcal{S}$-action $\rho$ on $G$ such that for each $v \in V(G)$ and $g \in \mathcal{S}$, $gp(v) = p(\rho(g)(v))$.
Unless otherwise stated, the group action corresponding to an
$\mathcal{S}$-symmetric framework will be free.

Let $\mathcal{S}$ be a subgroup of $\mathcal{E}(d)$.
An $\mathcal{S}$-gain graph $(G,\phi)$
along with a function $p: V(G)\rightarrow \mathbb{R}^d$ encodes an $\mathcal{S}$-symmetric framework $(H,q)$
on the covering graph $H$ of $(G,\phi)$ given by $q((g,v)) = gp(v)$.
Conversely, any $\mathcal{S}$-symmetric framework can be encoded by an $\mathcal{S}$-gain graph in this way.
In particular, if $(H,q)$ is $\mathcal{S}$-symmetric with corresponding $\mathcal{S}$-action $\rho$ and $U$ consists of one element from each vertex orbit, then we set $G := H\contract (\rho,U)$ and define $p:V(G)\rightarrow \mathbb{R}^d$
by $p(\mathcal{S}u) = q(u)$ for $u \in U$.
Given an $\mathcal{S}$-symmetric framework,
we refer to its representation as the pair $(H,q)$ as its \emph{classical representation}
and its representation as a triple $(G,\phi,p)$ as a \emph{gain graph representation} (note that gain graph representations need not be unique).
Figure~\ref{fig:representAsGainGraph} shows the classical and gain graph
representations of an $\mathcal{S}$-symmetric framework
where $\mathcal{S}$ is the group generated by a ninety degree rotation.
Figure~\ref{figure:wallpaperSymmetry} shows examples with wallpaper symmetry.

Given an integer $n$ and a group $\mathcal{S}$,
define $K_n(\mathcal{S})$ to be the directed graph on vertex set $\{1,\dots,n\}$ with
$|\mathcal{S}|-1$ loops at each vertex
and $|\mathcal{S}|$ parallel arcs from $i$ to $j$ whenever $i < j$.
Let $\psi_n$ be the gain function on $K_n(\mathcal{S})$ that associates a distinct
element of $\mathcal{S}$ to each non-loop edge of $K_n(\mathcal{S})$,
and a distinct non-identity element of $\mathcal{S}$ to each loop of $K_n(\mathcal{S})$.
We will often abuse notation and write $K_n(\mathcal{S})$ to mean the gain graph $(K_n(\mathcal{S}),\psi_n)$.
Figure~\ref{fig:completeGainGraph} shows $K_2(\mathbb{Z}_2)$.
When $\mathcal{S}$ is finite,
matroids of the form $\mathcal{M}(K_n(\mathcal{S}))$ are known as \emph{Dowling geometries} \cite{tanigawa2015matroids}.
When $\mathcal{S}$ is infinite, $\mathcal{M}(K_n(\mathcal{S}))$ is an infinite matroid with finite rank.

\begin{defn}\label{defn:symmetricCayleyMenger}
    Let $n,d$ be integers and let $\mathcal{S}$ be a subgroup of $\mathcal{E}(d)$.
    Define the map $D^\mathcal{S}_n: (\mathbb{R}^d)^n\rightarrow \mathbb{R}^{A(K_n(\mathcal{S}))}$
    so that for $e \in A(K_n(\mathcal{S}))$,
    \[
        D^\mathcal{S}_n(x^{(1)},\dots,x^{(n)})_e = \|x^{(\source(e))}-\psi_n(e)x^{(\target(e))}\|_2^2.
    \]
    The \emph{$\mathcal{S}$-symmetry-forced Cayley-Menger variety} $\cm_n^\mathcal{S}$
    is the Zariski closure of $D^\mathcal{S}_n((\mathbb{R}^d)^n)$.
\end{defn}

When $\mathcal{S}$ is a subgroup of $\mathcal{E}(d)$,
each $\mathcal{S}$-gain graph $(G,\phi)$ on vertex set $\{1,\dots,n\}$ can be viewed as a subset
of the coordinates of $\cm_n^\mathcal{S}$.
With this in mind, we say that $(G,\phi)$ is \emph{generically infinitesimally ($\mathcal{S}$-symmetry forced) rigid}
if $(G,\phi)$ is a spanning set of $\mathcal{M}(\cm_n^\mathcal{S})$.
Strictly speaking, this notion of generic infinitesimal symmetry forced rigidity generalizes
other notions in the literature (e.g.~\cite{jordan2016gain,ross2011geometricDissertation,schulze2011orbit}),
but the idea is essentially the same in all cases,
the only differences coming from restrictions on the group $\mathcal{S}$.
For the reader not familiar with this literature, we now defend this terminology.

An \emph{$\mathcal{S}$-symmetric motion} of a symmetric framework $(G,\phi,p)$ is a continuous function
$f: [0,1]\rightarrow (\mathbb{R}^d)^{V(G)}$ such that
$\|f(t)^{(\source(e))}-\phi(e)f(t)^{(\target(e))}\|_2^2 = \|p^{(\source(e))}-\phi(e)p^{(\target(e))}\|_2^2$
for all $e \in A(G)$ and $t \in [0,1]$.
Any $\mathcal{S}$-symmetric motion of $(G,\phi,p)$ that is also a motion of
$(K_n(\mathcal{S}),\psi_n,q)$ for all $q: V(G)\rightarrow \mathbb{R}^d$ is called \emph{trivial}.
An infinitesimal motion of $(G,\phi,p)$ is a vector $g \in (\mathbb{R}^d)^{V(G)}$
such that $\langle g^{(\source(e))}-\phi(e)g^{(\target(e))}, p^{(\source(e))}-\phi(e)p^{(\target(e))}\rangle = 0$
for all $e \in A(G)$.
If $f$ is a motion of $(G,\phi,p)$, then $f'(0)$ is an infinitesimal motion of $(G,\phi,p)$.
When $g = f'(0)$ for a trivial motion $f$ of $(G,\phi,p)$, then $g$ is said to be a \emph{trivial infinitesimal motion}.
The set of infinitesimal motions of a symmetric framework $(G,\phi,p)$ is a linear subspace of $(\mathbb{R}^d)^{V(G)}$
that always contains the set of trivial infinitesimal motions.
One says that $(G,\phi,p)$ is \emph{$\mathcal{S}$-symmetry forced infinitesimally rigid} if 
all its $\mathcal{S}$-symmetric infinitesimal motions are trivial.
Infinitesimal rigidity of $(G,\phi,p)$ is thus equivalent 
to the condition that its space of $\mathcal{S}$-symmetric infinitesimal motions has
the same dimension as the linear space of trivial infinitesimal motions.

It is a straightforward computation to see that the set of infinitesimal motions of $(G,\phi,p)$ is the nullspace
of the differential of $\pi_{G,\phi}\circ D_n^\mathcal{S}$ at $p$.
Thus $(G,\phi,p)$ is infinitesimally rigid when the submatrix of the differential of $\pi_{G,\phi}\circ D_n^\mathcal{S}$
has maximum possible rank as $(G,\phi)$ and $p$ are allowed to vary.
For fixed $(G,\phi)$, the rank of this differential is maximized at any $p$ such that
at least one maximal non-identically-zero minor of the Jacobian of $\pi_{G,\phi}\circ D_n^\mathcal{S}$
does not vanish when $p$ is plugged in.
The set of such $p$ is a Zariski-open subset of $(\mathbb{R}^d)^n$,
i.e.~such $p$ are \emph{generic}.
The dimension of the image of a polynomial map is equal to the rank of its differential at a generic point,
so $(G,\phi)$ is spanning in $\mathcal{M}(\cm_n^\mathcal{S})$ if and only if $(G,\phi,p)$
is infinitesimally rigid whenever $p$ is generic.
Thus it makes sense to call spanning sets of $\mathcal{M}(\cm_n^\mathcal{S})$ generically infinitesimally ($\mathcal{S}$-symmetry forced) rigid.

\subsection{Cyclic symmetry groups}
One of the main results of \cite{jordan2016gain,malestein2015frameworks} is a Laman-like characterization
of the minimally generically infinitesimally rigid $\mathcal{S}$-gain graphs when $\mathcal{S}$ is a finite
rotation subgroup of $\mathcal{E}(2)$.
We now offer a short proof of this result using Theorem~\ref{thm:mainResult}.

\begin{thm}[\cite{jordan2016gain,malestein2015frameworks}]\label{thm:cyclicSymmetryMatroid}
    Let $\mathcal{S}\subset\mathcal{E}(2)$ be a finite rotation subgroup.
    Then $\mathcal{M}(\cm_n^\mathcal{S}) = \mathcal{M}(2r_{\mathcal{M}(K_n(\mathcal{S}))}-1)$.
    In particular, if $|\mathcal{S}| > 1$, then an $\mathcal{S}$-gain graph $(G,\phi)$ is minimally generically infinitesimally rigid if and only if
    $G$ has $2|V(G)|-1$ edges, and for all subgraphs $G'$ of $G$,
    \[
        |E(G')| \le -1 + \sum_{F \in C(E(G'))} 2|V(F)|-2\alpha(F)
    \]
    where $\alpha(F) = 0$ when the sub gain graph of $(G,\phi)$
    with edge set $F$ contains an unbalanced cycle and $\alpha(F) = 1$ otherwise.
\end{thm}
\begin{proof}
    Theorem 2.1(j) in~\cite{zaslavsky1991biased} tells us that
    the rank function of the frame matroid $\mathcal{M}(G,\phi)$ is $r_{\mathcal{M}(G,\phi)}(S) = \sum_{F \in C(S)} V(F) - \alpha(F)$.
    Thus it now suffices to prove the first part of the theorem, i.e.~that
    $\mathcal{M}(\cm_n^\mathcal{S}) = \mathcal{M}(2r_{\mathcal{M}(K_n(\mathcal{S}))}-1)$.
    Since $\mathcal{S}$ is a finite rotation subgroup of $\mathcal{E}(2)$, we can express it as
    \[
        \mathcal{S} = \left\{\begin{pmatrix}
            \cos(\theta_k) & -\sin(\theta_k)\\
            \sin(\theta_k) & \cos(\theta_k)
        \end{pmatrix}:
        \theta_k = \frac{2k\pi}{|\mathcal{S}|}, k = 1,\dots,|\mathcal{S}|
        \right\}.
    \]
    Now, we may write
    \[
        \psi_n(e) = \begin{pmatrix}
            \cos(\theta_e) & -\sin(\theta_e)\\
            \sin(\theta_e) & \cos(\theta_e)
        \end{pmatrix}.
    \]
    If we write the coordinates of the domain $(\mathbb{C}^2)^n$ of $D_n^\mathcal{S}$ as $(x_i,y_i)_{i=1}^n$,
    then the parameterization $D_n^\mathcal{S}$ of $\cm_n^\mathcal{S}$ tells us that a generic
    point $d \in \cm_n^\mathcal{S}$ can be expressed as
    \[
        d_e = (x_{\source(e)} - \cos(\theta_e)x_{\target(e)} + \sin(\theta_e)y_{\target(e)})^2 +
        (y_{\source(e)} - \sin(\theta_e)x_{\target(e)} - \cos(\theta_e)y_{\target(e)})^2.
    \]
    We apply the following change of variables (letting $i$ denote $\sqrt{-1}$)
    \[
        x_u \mapsto \frac{x_u + y_u}{2} \qquad y_u \mapsto \frac{x_u - y_u}{2i},
    \]
    so that our parameterization becomes
    \[
        d_e = (x_{\source(e)} - \exp(i\theta_e) x_{\target(e)})(y_{\source(e)} - \exp(-i\theta_e) y_{\target(e)}).
    \]
    Thus we have expressed $\cm_n^\mathcal{S}$ as the Hadamard product of two linear spaces
    $L,L'\subseteq \mathbb{C}^{A(K(\mathcal{S}))}$
    respectively parameterized as $d_{e} = x_{\source(e)} - \exp(i\theta_e) x_{\target(e)}$
    and $d_e = y_{\source(e)} - \exp(-i\theta_e) y_{\target(e)}$.
    Theorem~\ref{thm:mainResult} now implies that it suffices to show that
    $\mathcal{M}(L) = \mathcal{M}(L') = \mathcal{M}(K_n(\mathcal{S}))$.
    This follows from \cite[Lemma 6.10.11]{oxley2006matroid} by noting that all the loop edges
    in $K_n^\mathcal{S}$ incident to a single vertex form a parallel class.
\end{proof}

\section{Symmetry groups with translations}\label{section:symmetryGroupsWithTranslation}
The punchline of this section is Theorem~\ref{thm:mainTheoremTranslationsAndRotations} which gives a combinatorial characterization
of symmetry-forced infinitesimal rigidity in the case that the symmetry group is a subgroup of $\mathbb{R}^2\rtimes SO(2)$,
i.e.~consists of rotations and translations.
The more technical parts of this section are powered by an extremely lucky coincidence,
namely that the action of $SO(2)$ on $\mathbb{R}^2$ can be modeled by complex arithmetic.
The corresponding symmetric Cayley-Menger varieties can be expressed as Hadamard products of \emph{affine} spaces,
but not linear spaces.
Thus, our first order of business is Theorem~\ref{thm:mainAffine},
a generalization of Theorem~\ref{thm:mainResult} that can handle affine spaces.
For this, we will require basic results about elementary lifts of matroids.

\subsection{Elementary lifts of matroids}
A set of circuits $\mathcal{C}$ of a matroid $M$ is called a \emph{linear class}
if whenever $C_1 \cup C_2 \in \mathcal{C}$ satisfy $r_M(C_1 \cup C_2) = |C_1 \cup C_2|-2$,
then $C_3 \in \mathcal{C}$ whenever $C_3$ is a circuit satisfying $C_3 \subset C_1 \cup C_2$.
Given a matroid $M$ on ground set $E$ and a linear class of circuits $\mathcal{C}$,
a set $S \subseteq E$ is \emph{$\mathcal{C}$-balanced} if every circuit contained in $S$ is in $\mathcal{C}$.

An \emph{elementary coextension} of a matroid $M$ is a matroid $N$ such that $N\contract \{e\} = M$
for an element $e$ in the ground set of $N$.
An \emph{elementary lift} of $M$ is a matroid of the form $N\setminus\{e\}$
such that $M = N\contract\{e\}$.
The following proposition tells us that the elementary lifts of a given matroid are characterized by its linear classes.

\begin{prop}[{\cite[Proposition 3.1]{whittle1989generalisation}}]\label{prop:liftsFromLinearClasses}
    Let $M$ be a matroid on ground set $E$ and let $\mathcal{C}$ be a linear class of circuits of $M$.
    Then the function $r: 2^E\rightarrow \mathbb{Z}$ defined by
    \[
        r_\mathcal{C}(S) := \begin{cases}
            r_M(S) & \textnormal{if} \ S \ \textnormal{is } \mathcal{C}\textnormal{-balanced} \\
            1+r_M(S) & \textnormal{otherwise}
        \end{cases}
    \]
    is the rank function of a lift of $M$.
    Moreover, every lift of $M$ can be obtained in this way.
\end{prop}

Given a matroid $M$ and a linear class $\mathcal{C}$ of circuits of $M$,
we let $M^\mathcal{C}$ denote the elementary lift of $M$ corresponding to $\mathcal{C}$
as in Proposition~\ref{prop:liftsFromLinearClasses}.

\subsection{Algebraic matroids of Hadamard products of affine spaces}
The set of matrices with entries in a field $\mathbb{K}$ whose rows are
indexed by a set $E$ and columns indexed by a set $F$ will be denoted $\mathbb{K}^{E\times F}$.
When $E$ or $F$ is $\{1,\dots,r\}$, we will simply write $\mathbb{K}^{E\times r}$ or $\mathbb{K}^{r\times F}$.

Let $V\subseteq \mathbb{C}^{E}$ be an affine space.
Then there exists $A(V) \in \mathbb{C}^{E\times r}$ and $b(V) \in \mathbb{C}^{E}$ such that $b(V)^TA(V) = 0$
and $V = \{A(V)x + b(V): x \in \mathbb{C}^r\}$.
Since $b(V)^TA(V) = 0$, $b(V)$ is uniquely defined.
Only the column span of $A(V)$ matters to us, so the fact that $A(V)$ is not uniquely defined will not be a problem.
Define the linear space $L(V)\subseteq \mathbb{C}^{E\cup \{*\}}$ as follows
\begin{equation}\label{eq:L(V)}
    L(V) := \left\{
    \begin{pmatrix}
        A(V) & b(V) \\ 0 & 1
    \end{pmatrix}
    \begin{pmatrix}
        x \\ \lambda
    \end{pmatrix}:
    x \in \mathbb{C}^r, \lambda \in \mathbb{C}
    \right\}.
\end{equation}
Note that $\mathcal{M}(V)$ is the matroid of linear independence on the rows of $A(V)$,
and $\mathcal{M}(L(V))$ is the matroid of linear independence on the rows of the block matrix
defining $L(V)$ in \eqref{eq:L(V)}.
It then follows that $\mathcal{M}(L(V))$ is an elementary coextension of $\mathcal{M}(V)$.
The corresponding elementary lift will play such an important role that we give it its own notation.

\begin{defn}
    Given an affine space $V\subseteq \mathbb{C}^E$,
    we denote the elementary lift $\mathcal{M}(L(V))\setminus \{*\}$ of $\mathcal{M}(V)$
    by $\mathcal{M}^L(V)$
    and the corresponding linear class of $\mathcal{M}(V)$ by $\mathcal{C}(V)$.
\end{defn}

The following lemma gives an explicit description of $\mathcal{C}(V)$.
Given a subset $S$ of the rows of a matrix $M$, we let $M_S$ denote the row-submatrix of $M$ corresponding to $S$.

\begin{lemma}\label{lemma:liftOfMatroidOfAffineSpace}
    Let $V \subseteq \mathbb{C}^E$ be an affine space
    and let $\mathcal{C}$ consist of the circuits $C$ of $\mathcal{M}(V)$
    such that $b(V)_C \in {\rm span}(A(V)_C)$.
    Then $\mathcal{C}(V) = \mathcal{C}$.
\end{lemma}
\begin{proof}
    This follows from Proposition~\ref{prop:liftsFromLinearClasses}
    by noting the presentation of $L(V)$ in~\eqref{eq:L(V)}.
\end{proof}

For $\alpha \in \mathbb{C}$,
define $H_\alpha\subseteq \mathbb{C}^{E \cup \{*\}}$ by $H_\alpha := \{x \in \mathbb{C}^{E \cup \{*\}}: x_{*} = \alpha\}$.
Then $V = \pi_E(H_1 \cap L(V))$.
Letting $r_L$ and $r_V$ denote the rank functions of $\mathcal{M}(L(V))$ and $\mathcal{M}(V)$, we have
\[
    r_L(S) = \begin{cases}
        r_V(S) & \textnormal{if } * \notin S \textnormal{ and } b(V)_S \in {\rm span}(A(V)_S) \\
        1+r_V(S\setminus \{*\}) & \textnormal{otherwise}.
    \end{cases}
\]

\begin{lemma}\label{lemma:tropOfAffine}
    Let $E$ be a finite set and let $V \subseteq \mathbb{C}^E$ be an affine space.
    Then $\trop(V) = \pi_E(\trop(L(V)) \cap H_0)$.
\end{lemma}
\begin{proof}
    Since $\trop(H_1) = H_0$
    it suffices to show that $\trop(L(V)) \cap H_0$ is the \emph{stable intersection}
    of $\trop(L(V))$ and $H_0$ (\cite{osserman2013lifting}, but see \cite{jensen2016stable} for a more elementary presentation).
    By \cite[Lemma~2.6]{jensen2016stable}, it is enough to note that
    for any cone $\sigma$ of $\trop(L(V))$, the Minkowski sum $\sigma + H_0$ is $(|E|+1)$-dimensional.
    To see that this is true, note that if $\sigma + H_0$ is not $(|E|+1)$-dimensional,
    then $\sigma \subseteq H_0$. But this is impossible because the lineality space of
    every cone in $\trop(L(V))$ is spanned by the all-ones vector.
\end{proof}

\begin{lemma}\label{lemma:affineSpaceProduct}
    If $U,V\subseteq \mathbb{C}^E$ are finite-dimensional affine spaces, then
    \[
        \mathcal{M}(U\star V) = \mathcal{M}(L(U)\star L(V))\contract \{*\}.
    \]
\end{lemma}
\begin{proof}
    Assume that $E$ is finite.
    Lemmas~\ref{lemma:HadamardProductMinkowskiSum} and~\ref{lemma:tropOfAffine},
    and Proposition~\ref{prop:algMatroidOfLinearSpaceIsBergmanFanOfMatroid} imply
    \[
        \trop(U \star V) = \pi_E(H_0 \cap (\berg(\mathcal{M}(L(U))) + \berg(\mathcal{M}(L(V))))).
    \]
    It follows from this that $I$ is independent in $\mathcal{M}(U\star V)$
    if and only if $I \cup\{*\}$ is independent in $\mathcal{M}(L(U) \star L(V))$.
    
    Now assume $E$ is infinite.
    Then $I$ is independent in $\mathcal{M}(U\star V)$ if and only if
    $I$ is independent in $\mathcal{M}(U\star V)|I$.
    Since $\mathcal{M}(U\star V)$ has finite rank, the result now follows from the finite case.
\end{proof}

\begin{lemma}\label{lemma:contractSubmodular}
    Let $f: 2^E\rightarrow \mathbb{Z}$ be increasing, submodular, and bounded.
    For $e \in E$, define $g_e: 2^{E\setminus e}\rightarrow \mathbb{Z}$ by $g_e(S) = \min\{f(S),f(S\cup \{e\})-1\}$.
    Then $\mathcal{M}(f)\contract e = \mathcal{M}(g_e)$.
\end{lemma}
\begin{proof}
    A subset $I \subseteq E$ is independent in $\mathcal{M}(f)\contract e$ if and only if
    $I \cup \{e\}$ is independent in $\mathcal{M}(f)$.
    Thus for any subset $I' \subseteq I$,
    we must have $|I'| \le f(I')$ and $|I' \cup \{e\}| \le f(I' \cup \{e\})$.
    Both conditions are satisfied precisely when $|I'| \le g_e(I')$.
\end{proof}


\begin{thm}\label{thm:mainAffine}
    Let $U,V \subseteq \mathbb{C}^E$ be finite-dimensional affine spaces
    and let $M = \mathcal{M}(U)$ and $N = \mathcal{M}(V)$.
	Then $\mathcal{M}(U\star V) = \mathcal{M}(f)$ where $f:2^{E}\rightarrow \mathbb{Z}$ is defined by
    \[
        f(S) = \begin{cases}
            r_M(S) + r_N(S) -1 & \textnormal{if } S \textnormal{ is } \mathcal{C}(U)\textnormal{-balanced and } \mathcal{C}(V)\textnormal{-balanced}\\
            r_M(S) + r_N(S) & \textnormal{otherwise}.
        \end{cases}
    \]
\end{thm}
\begin{proof}
    Let $t_U$ and $t_V$ denote the rank functions of $\mathcal{M}(L(U))$ and $\mathcal{M}(L(V))$.
    Lemmas~\ref{lemma:affineSpaceProduct} and~\ref{lemma:contractSubmodular} imply that
    $\mathcal{M}(U\star V) = \mathcal{M}(g)$
    where $g: 2^E\rightarrow \mathbb{Z}$ is defined by
    \[
        g(S) = \min\{t_U(S) + t_V(S) - 1,t_U(S\cup \{*\}) + t_V(S \cup \{*\}) - 2\}.
    \]
	Since $\mathcal{M}(L(Z))\contract \{*\} = \mathcal{M}(Z)$ for $Z = U,V$, this simplifies to
	\[
		g(S) = \min\{t_U(S) + t_V(S) - 1,r_M(S) + r_N(S)\}.
	\]
    Note that $t_U(S) + t_V(S) \ge r_M(S) + r_N(S)$ with equality precisely when
    $b(Z)_S \in {\rm span}(A(Z)_S)$ for $Z = U$ and $Z = V$.
    Lemma~\ref{lemma:liftOfMatroidOfAffineSpace} implies that 
    this is equivalent to $S$ being $\mathcal{C}(U)$- and $\mathcal{C}(V)$-balanced.
    Therefore $f = g$.
\end{proof}



\subsection{\texorpdfstring{Matroids for $\mathbb{R}^2\rtimes SO(2)$-gain graphs}{Matroids for R2 x| SO(2)-gain graphs-gain graphs}}
This section contains the specific details that are needed to state and prove our characterization
of generic symmetry-forced rigidity for two dimensional frameworks whose symmetry groups consist of rotations and translations.
Definition~\ref{defn:matricesForSemidirectProduct} defines two matrices that will be used to prove
Theorem~\ref{thm:mainTheoremTranslationsAndRotations}.
Lemma~\ref{lemma:mainLemma} gives the precise relationship between their matroids;
it is the technical meat of our main theorem.
The glue that holds this lemma together is Remark~\ref{remark:luckyCoincidence},
an extremely lucky coincidence that lets us model the behavior of our symmetry group using complex arithmetic.

Recall that the Euclidean group $\mathcal{E}(d)$ can be expressed as the semidirect product
$\mathcal{E}(d) = \mathbb{R}^d \rtimes O(d)$.
More explicitly,
the composition rule is $(x_1,A_1)(x_2,A_2) = (x_1 + A_1x_2,A_1A_2)$
and the inverse of $(x,A)$ is $(-A^{-1}x,A^{-1})$.
The projection maps onto $\mathbb{R}^d$ and $O(d)$ will be respectively denoted $\pi_1$ and $\pi_2$.
Note that only $\pi_2$ is a group homomorphism.
If $\mathcal{S}$ is a subgroup of $O(d)$, then $\mathbb{R}^d \rtimes \mathcal{S}$ is a subgroup of $\mathcal{E}(d)$.

Let $(G,\phi)$ be a gain graph whose gain group is a subgroup of $\mathbb{R}^2\rtimes SO(2)$.
Then $(G,\pi_2\circ \phi)$ is an $SO(2)$-gain graph.
The goal of the rest of this subsection is to define the matroid $\mathcal{M}^L(G,\phi)$
which will be an elementary lift of $\mathcal{M}(G,\pi_2\circ\phi)$.
In light of Proposition~\ref{prop:liftsFromLinearClasses}, it will be enough to indicate which circuits
of $\mathcal{M}(G,\pi_2\circ\phi)$ are also circuits of $\mathcal{M}^L(G,\phi)$.

A graph is \emph{bicyclic} if it is a subdivision of one of the three graphs shown in Figure~\ref{figure:bicyclicGraphs}.
Subdivisions of the first are called \emph{tight handcuffs}, of the second \emph{loose handcuffs},
and of the third \emph{theta graphs}.
The circuits $\mathcal{M}(G,\phi)$ for any gain graph consist of the balanced cycles
and the bicyclic subgraphs with no induced balanced cycle~\cite[Theorem 6.10.5]{oxley2006matroid}.

\begin{figure}
    \begin{tikzpicture}
        \vertex (a) at (0,0){};
        \draw (a) to [out=45,in=-45,looseness=20] (a);
        \draw (a) to [out=135,in=-135,looseness=20] (a);
    \end{tikzpicture} \qquad
    \begin{tikzpicture}
        \vertex (a) at (0,0){};
        \vertex (b) at (1,0){};
        \path (a) edge (b);
        \draw (b) to [out=45,in=-45,looseness=20] (b);
        \draw (a) to [out=135,in=-135,looseness=20] (a);
    \end{tikzpicture} \qquad
    \begin{tikzpicture}
        \vertex (a) at (0,0){};
        \vertex (b) at (1,0){};
        \path (a) edge (b);
        \draw (a) to [out=45,in=135] (b);
        \draw (a) to [out=-45,in=-135] (b);
        \draw[color=white] (b) to [out=45,in=-45,looseness=20] (b);
    \end{tikzpicture}
    \vspace{-15pt}
    \caption{Any graph that is a subdivision of one of these graphs is \emph{bicyclic}.
    Subdivisions of the first are called \emph{tight handcuffs}, of the second \emph{loose handcuffs}
    and of the the third \emph{theta graphs}.}\label{figure:bicyclicGraphs}
\end{figure}
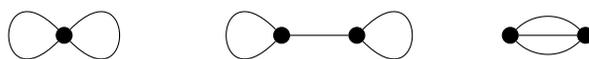

\begin{defn}\label{defn:dutch}
    Let $(G,\phi)$ be an $\mathcal{S}$-gain graph and let $B$ be a bicyclic subgraph.
    A \emph{covering pair of walks} is a pair $(W_1,W_2)$ of closed walks in $B$,
    based at the same vertex,
    such that every edge of $B$ is visited between one and two times by the concatenation $W_1W_2$.
    We say that $B$ is \emph{Dutch} if there exists a covering pair $(W_1,W_2)$ of walks such that $\phi(W_1)\phi(W_2) = \phi(W_2)\phi(W_1)$.
\end{defn}

\begin{lemma}\label{lemma:dutch}
  A bicyclic subgraph $B$ of $(G,\phi)$ is Dutch if and only if for all covering pairs $(W_1,W_2)$ of walks,
  $\phi(W_1)\phi(W_2) = \phi(W_2)\phi(W_1)$.
\end{lemma}
\begin{proof}
    The ``if'' direction is trivial.
    To prove the ``only if'' direction, let $(W_1,W_2)$ and $(W_1',W_2')$ be covering pairs of walks
    and assume that $\phi(W_1)\phi(W_2) = \phi(W_2)\phi(W_1)$.
    Let $v$ be the basepoint of $W_i$ and $v'$ of $W_i'$.
    If $v = v'$, then $\{\phi(W_i),\phi(W_i)^{-1} : i = 1,2\} = \{\phi(W_i'),\phi(W_i')^{-1} : i = 1,2\}$.
    If two elements of a group commute with each other,
    then they commute with each others' inverses and so we may assume $v \neq v'$.
    This means that $B$ is not a tight handcuff.
    We may then without loss of generality assume that $v$ and $v'$ have degree at least three
    (following the convention that each loop edge contributes two to the degree of its incident vertex),
    since if either has degree two, then we may contract one of the incident edges
    and compose its gain with the gain of the other incident edge
    without changing $\phi(W_i)$ or $\phi(W_i')$.
    This means that $B$ is either the second or third graph pictured in Figure~\ref{figure:bicyclicGraphs}
    with vertex set $\{v,v'\}$.
    In either case, it is easy to check that $\phi(W_1)$ and $\phi(W_2)$ commute
    if and only if $\phi(W_1')$ and $\phi(W_2')$ commute.
\end{proof}

\begin{thm}[Main theorem]\label{thm:mainTheoremTranslationsAndRotations}
    Let $\mathcal{S}$ be a subgroup of $\mathbb{R}^2\rtimes SO(2)$
    and let $(G,\phi)$ be an $\mathcal{S}$-gain graph with $n\ge 2$ vertices.
    Define $\alpha:2^{A(G)}\rightarrow \{0,1,2,3\}$ as follows
    \[
        \alpha(H) = \begin{cases}
            3 & \textnormal{if every cycle in } H \textnormal{ is balanced}\\
            2 & \textnormal{if not, and the gain of each cycle is a translation}\\
            1 & \textnormal{if none of the above, and all bicyclic subgraphs are Dutch}\\
            0 & \textnormal{otherwise}.
        \end{cases}
    \]
    Then $(G,\phi)$ is generically minimally infinitesimally rigid in $\mathbb{R}^2$ if and only if
    $|A(G)| = 2|V(G)|-\alpha(A(K_{n}(\mathcal{S})))$ and $|F| \le 2|V(F)| - \alpha(F)$
    for all $F \subseteq A(G)$.
\end{thm}

Inferring $\alpha(A(K_n(\mathcal{S})))$ from $\mathcal{S}$ is easy:~if $n \ge 2$, then $\alpha(A(K_n(\mathcal{S}))) = 0$ when $\mathcal{S}$ is non-Abelian,
$\alpha(A(K_n(\mathcal{S}))) = 1$ when $\mathcal{S}$ is a rotation group,
$\alpha(A(K_n(\mathcal{S}))) = 2$ when $\mathcal{S}$ is a translation group,
and $\alpha(A(K_n(\mathcal{S}))) = 3$ when $\mathcal{S}$ is trivial.
As a sanity check for Theorem~\ref{thm:mainTheoremTranslationsAndRotations},
note that Theorems~\ref{thm:cyclicSymmetryMatroid},~\ref{thm:lamansMatroid}, and \cite[Theorem 5.2]{ross2015inductive} are immediate consequences.

\begin{ex}\label{ex:wallpaperGroup}
    We will now consider a symmetric framework whose symmetry group is the wallpaper
    group with one degree-four rotation and no reflections nor glide-reflections.
    Let $\{e_1,e_2\}$ denote the standard basis of $\mathbb{R}^2$
    and let $A$ denote the matrix of a counterclockwise rotation ninety degrees about the origin.
    Let $\mathcal{S}$ denote the subgroup of $\mathcal{E}(2)$ generated by $(e_1,I)$, $(e_2,I)$ and $(0,A)$.
    Figure~\ref{figure:wallpaperSymmetry} shows an $\mathcal{S}$-symmetric framework,
    in its classical representation and its gain graph representation.
    Theorem~\ref{thm:mainTheoremTranslationsAndRotations} tells us that the graph underlying this framework is \emph{not}
    generically rigid.
    To see this, note that every cycle in the sub gain graph on vertices $v_1$ and $v_2$ is unbalanced
    with a translation gain. If the entire framework were rigid, this subgraph should have at most two
    edges, but this is not the case.
    If we replace the gain labels on the loops at $v_1$ and $v_2$ by $(e_1,A)$ and $(e_2,A)$,
    then the resulting framework is rigid
    since then every induced bicyclic subgraph would be non-Dutch.
    This example cannot be handled by any of the previously quoted theorems in this paper, but it can be handled by the main results of~\cite{malestein2014frameworks}.
    \begin{figure}
        \includegraphics[scale=0.46]{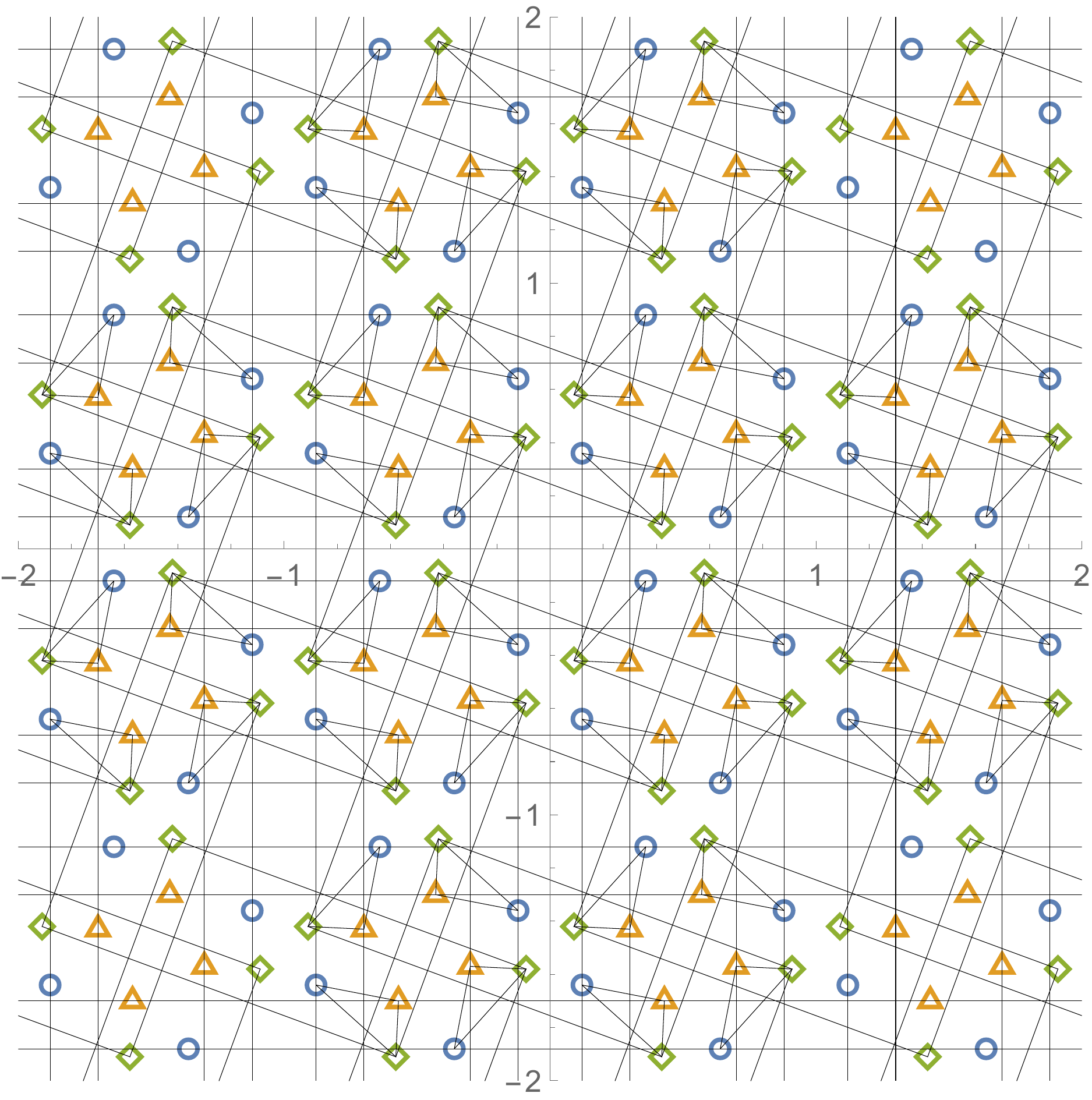}
        \begin{tikzpicture}
            \node (a) at (0,0){\includegraphics[scale=0.01]{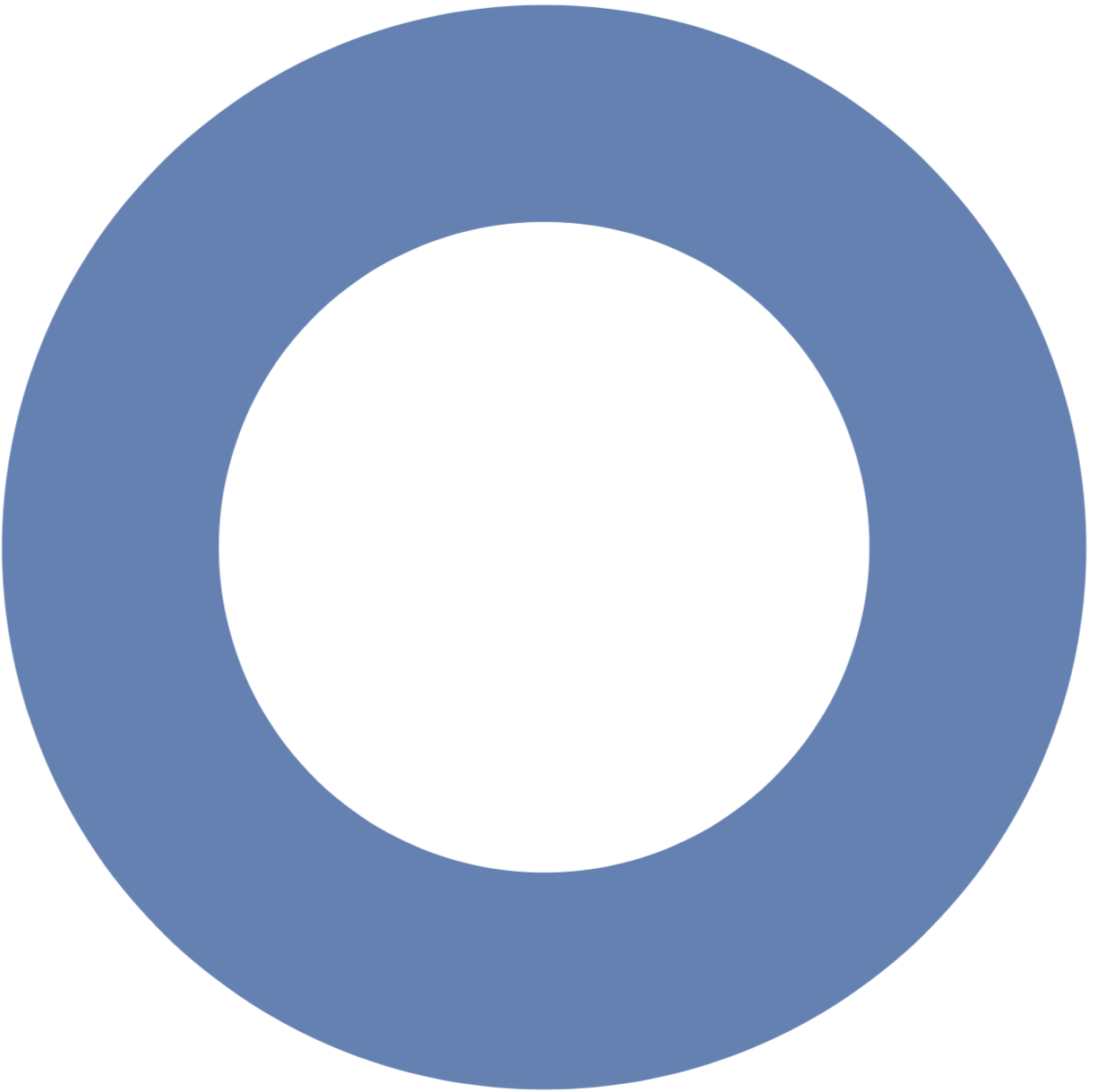}};
            \node (b) at (4,0){\includegraphics[scale=0.01]{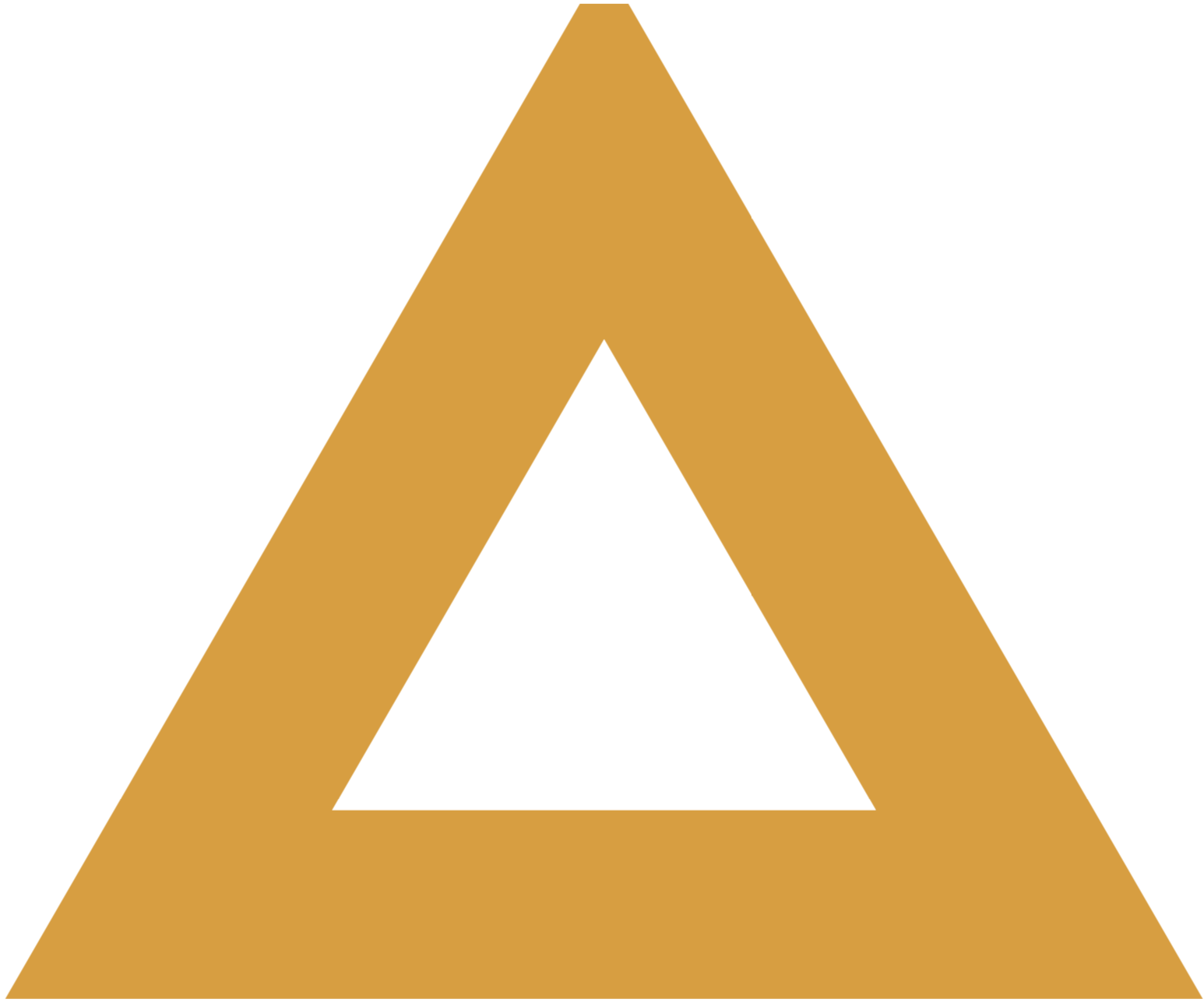}};
            \node (c) at (2,2){\includegraphics[scale=0.01]{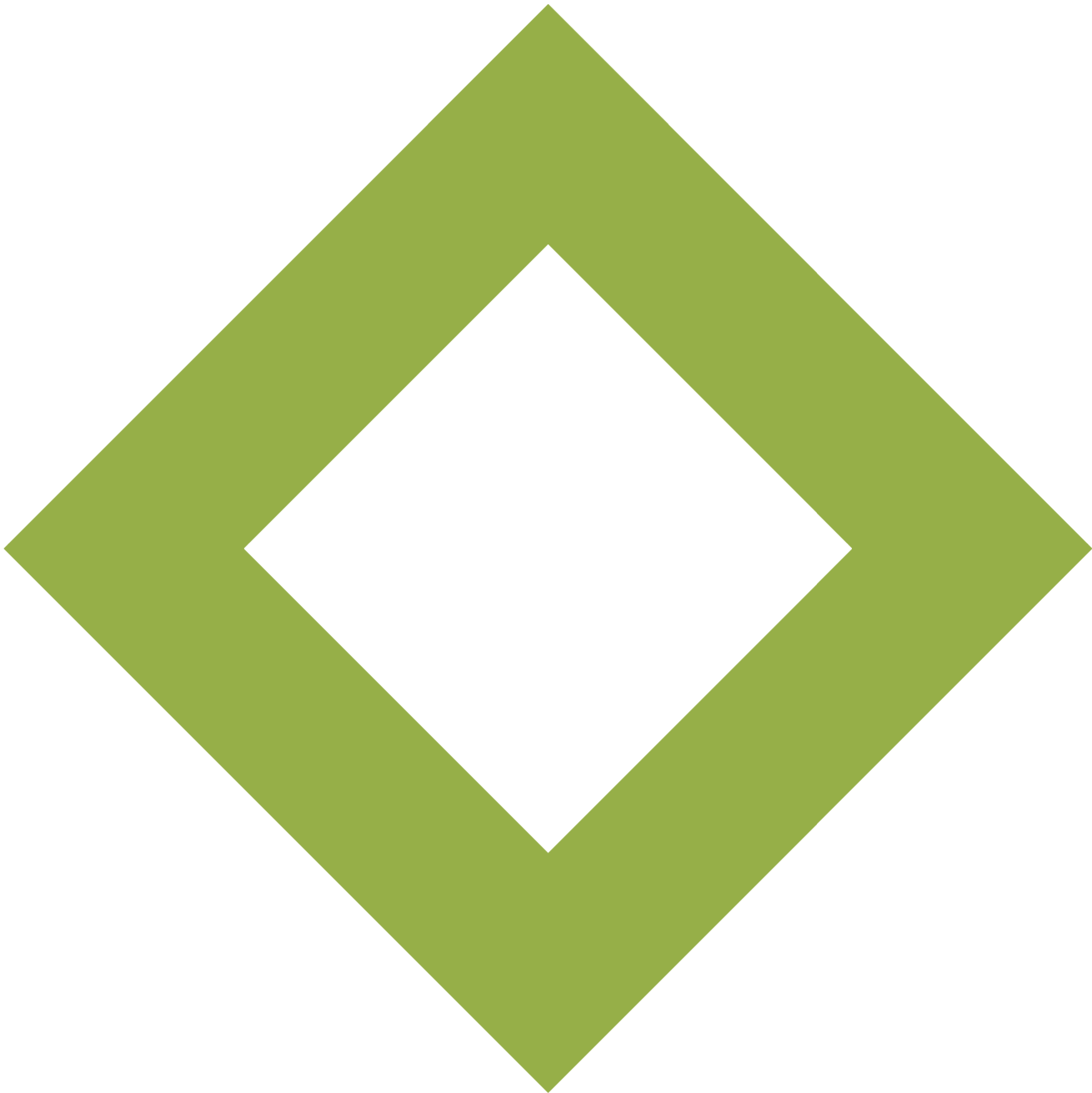}};
            \node at (.6,.2){$v_1$};
            \node at ({4-.6},.2){$v_2$};
            \node at (2,{2-.4}){$v_3$};
            \node at (0,-1){$\left(e_1,I\right)$};
            \node at (4,-1){$\left(e_2,I\right)$};
            \node at (2,3){$\left(0,A\right)$};
            \node at (0,-3.5){};
            \draw[->] (a) to [out=330,in=210,looseness=5] (a);
            \draw[->] (b) to [out=330,in=210,looseness=5] (b);
            \draw[->] (c) to [out=30,in=150,looseness=5] (c);
            \path[->]
                (a) edge (b) edge (c)
                (b) edge (c)
            ;
        \end{tikzpicture}
        \\
        \includegraphics[scale=0.46]{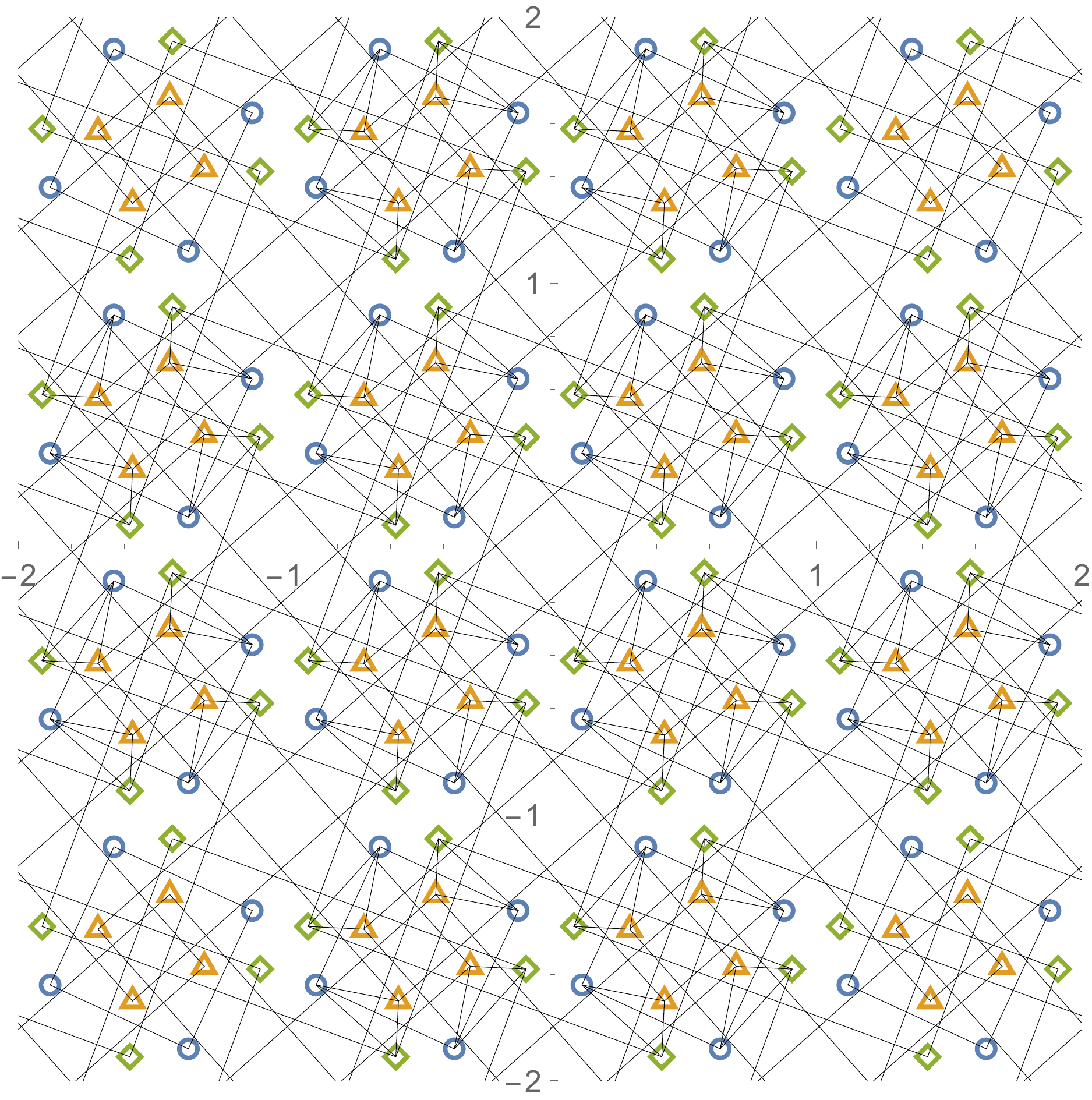}
        \begin{tikzpicture}
            \node (a) at (0,0){\includegraphics[scale=0.01]{blueCircle.png}};
            \node (b) at (4,0){\includegraphics[scale=0.01]{orangeTriangle.png}};
            \node (c) at (2,2){\includegraphics[scale=0.01]{greenDiamond.png}};
            \node at (.6,.2){$v_1$};
            \node at ({4-.6},.2){$v_2$};
            \node at (2,{2-.4}){$v_3$};
            \node at (0,-1){$\left(e_1,A\right)$};
            \node at (4,-1){$\left(e_2,A\right)$};
            \node at (2,3){$\left(0,A\right)$};
            \node at (0,-3.5){};
            \draw[->] (a) to [out=330,in=210,looseness=5] (a);
            \draw[->] (b) to [out=330,in=210,looseness=5] (b);
            \draw[->] (c) to [out=30,in=150,looseness=5] (c);
            \path[->]
                (a) edge (b) edge (c)
                (b) edge (c)
            ;
        \end{tikzpicture}
        \caption{Two frameworks with wallpaper symmetry.
        Here, $A$ denotes the matrix of a ninety-degree counterclockwise rotation
        and $e_1,e_2$ denote the standard basis vectors of $\mathbb{R}^2$,
        and $p(v_1)$, $p(v_2)$, and $p(v_3)$ are respectively
        $(0.64,0.12)$, $(0.70,0.43)$, and $(0.91,0.42)$.
        The framework below is symmetry-forced rigid,
        whereas the framework above is not.
        One motion of the framework above that maintains the symmetry
        comes from rotating the points $Tp(v_1)$ and $Tp(v_2)$ counterclockwise about $Tp(v_3)$,
        where $T$ ranges over the symmetry group.
        An animation of this can be found at \url{https://dibernstein.github.io/Supplementary_materials/symRigid.html}.
        }\label{figure:wallpaperSymmetry}
    \end{figure}
\end{ex}

The remainder of this section proves Theorem~\ref{thm:mainTheoremTranslationsAndRotations}.
The main idea, which is fleshed out in Proposition~\ref{prop:mainTheoremTerseVersion},
is that when $\mathcal{S}\subset \mathbb{R}^2\rtimes SO(2)$,
we can apply Theorem~\ref{thm:mainAffine} since in this case
$\cm_n^\mathcal{S}$ is a Hadamard product of two affine spaces.
Definition~\ref{defn:matricesForSemidirectProduct} gives us matrices to describe these affine spaces,
and Lemma~\ref{lemma:mainLemma} describes their relevant combinatorial properties.

Let $\mathbb{T}$ denote the circle group, i.e.~the set of unit-modulus complex numbers under multiplication.
Define $t: SO(2)\rightarrow \mathbb{T}$ to be the group isomorphism
given by $t(A)=\exp(i\theta)$ where $\theta$ is the angle of rotation of the matrix $A$.
Define $c:\mathbb{R}^2 \rightarrow \mathbb{C}$ to be the group isomorphism given by $c(x,y) = x + iy$.

\begin{rmk}\label{remark:luckyCoincidence}
    Given $A \in SO(2)$ and $x \in \mathbb{R}^2$, $c(Ax) = t(A)c(x)$.
\end{rmk}

\begin{defn}\label{defn:matricesForSemidirectProduct}
    Let $(G,\phi)$ be an $\mathbb{R}^2\rtimes SO(2)$-gain graph.
    Define $M(G,\phi) \in \mathbb{C}^{A(G)\times V(G)}$ by
    \[
        M(G,\phi)_{e,v} :=
        \begin{cases}
            1 & \textnormal{when } e \textnormal{ is not a loop and } v=\source(e) \\
            -t(\pi_2(\phi(e))) & \textnormal{when } e \textnormal{ is not a loop and } v=\target(e) \\
            1-t(\pi_2(\phi(e))) & \textnormal{when } e \textnormal{ is a loop} \\
            0 & \textnormal{otherwise}.
        \end{cases}
    \]
    Define $M^L(G,\phi) \in \mathbb{C}^{A(G)\times (V(G)\cup \{*\})}$ by
    \[
        M^L(G,\phi)_{e,v} =
        \begin{cases}
            -c(\pi_1(\phi(e))) & \textnormal{when } v = * \\
            M(G,\phi)_{e,v} &\textnormal{when } v\neq *.
        \end{cases}
    \]
\end{defn}

\begin{rmk}\label{remark:switchingForNewConstruction}
    In Definition~\ref{defn:matricesForSemidirectProduct},
    the effect of reversing an arc $e$ of $G$ and inverting $\phi(e)$
    is to multiply the corresponding row of $M(G,\phi)$ and $M^L(G,\phi)$
    by $-t(\pi_2(\phi(e))^{-1})$.
    The statement for $M^L(G,\phi)$
    follows via Remark~\ref{remark:luckyCoincidence} from the formula for inverting an element of $\mathcal{E}(d)$.
\end{rmk}

\begin{ex}
    Figure~\ref{figure:liftMatrixExample} shows an $\mathbb{R}\rtimes SO(2)$-gain graph $(G,\phi)$ alongside $M^L(G,\phi)$.
\end{ex}

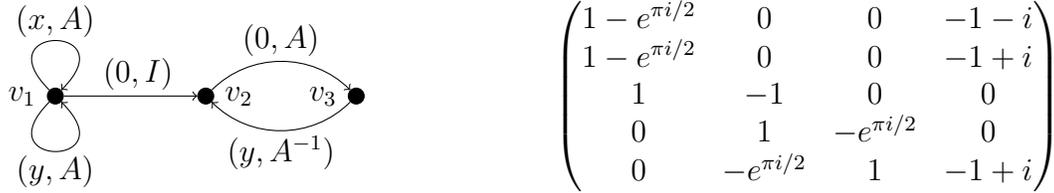
\begin{figure}[h]
    \begin{tikzpicture}
        \vertex (a) at (0,0)[label=left:$v_1$]{};
        \vertex (b) at (2,0)[label=right:$v_2$]{};
        \vertex (c) at (4,0)[label=left:$v_3$]{};
        \draw[->] (a) to [in=45,out=135,looseness=20] (a);
        \draw[->] (a) to [in=-45,out=-135,looseness=20] (a);
        \draw[->] (a) to (b);
        \draw[->] (b) to [out=45,in=135] (c);
        \draw[->] (c) to [out=-135,in=-45] (b);
        \node at (0,1){$(x,A)$};
        \node at (0,-1){$(y,A)$};
        \node at (3,.75){$(0,A)$};
        \node at (3,-.75){$(y,A^{-1})$};
        \node at (1.1,.3){$\left(0,I\right)$};
        \node at (10,0){$
        \begin{pmatrix}
        1-e^{\pi i/2} & 0 & 0 & -1-i\\
        1-e^{\pi i/2} & 0 & 0 & -1+i\\
        1 & -1 & 0 & 0\\
        0 & 1 & -e^{\pi i/2} & 0\\
        0 & -e^{\pi i/2} & 1 & -1+i\\
        \end{pmatrix}
    $};
    \end{tikzpicture}
    \caption{An $\mathbb{R}^2\rtimes SO(2)$-gain graph $(G,\phi)$ alongside $M^L(G,\phi)$.
    Here $A$ is the matrix of a $\pi/2$-radian counterclockwise rotation about the origin,
    $x = (1,1)^T$ and $y = (1,-1)^T$.}\label{figure:liftMatrixExample}
\end{figure}

Given a matrix $A \in \mathbb{K}^{E\times S}$,
the matroid given by linear independence on the \emph{rows} of $A$ will be denoted $\mathcal{M}(A)$.
This is the same matroid as $\mathcal{M}({\rm span}(A))$,
the algebraic matroid of the column span of $A$.
We say that a linear form $f:\mathbb{C}^E\rightarrow \mathbb{C}$ is \emph{supported on $S\subseteq E$}
if the nonzero coefficients of $f$
are at the coordinates indexed by $S$.
Given a matrix $B \in \mathbb{C}^{E\times F \cup \{*\}}$,
define $\aff(B) := \{Bx: x\in\mathbb{C}^{F \cup \{*\}} \textnormal{ with } x_{*} = 1\}$.
If $A$ is the column-submatrix of $B$ consisting of the columns indexed by $F$, then
$\mathcal{M}(\aff(B)) = \mathcal{M}(A)$.

\begin{lemma}\label{lemma:mainLemma}
    Let $(G,\phi)$ be an $\mathbb{R}^2\rtimes SO(2)$-gain graph
    and let $\mathcal{C}$ consist of all balanced circuits and all Dutch bicyclic subgraphs of $(G,\phi)$.
    Then
    \begin{enumerate}[(a)]
    	\item\label{item:isGainMatroid} $\mathcal{M}(\aff(M^L(G,\phi)))=\mathcal{M}(M(G,\phi)) = \mathcal{M}(G,\pi_2\circ\phi)$,
    	\item\label{item:isLinearClass} $\mathcal{C}$ is a linear class, and
    	\item\label{item:isTheLiftWeWant} $\mathcal{C}(\aff(M^L(G,\phi)))=\mathcal{C}$; in other words $\mathcal{M}(G,\pi_2\circ\phi)^{\mathcal{C}} = \mathcal{M}(M^L(G,\phi))$.
    \end{enumerate}
\end{lemma}
\begin{proof}
	The first equality in \ref{item:isGainMatroid} is clear and the second follows from \cite[Lemma~6.10.11]{oxley2006matroid}
    since the loops at a given vertex in $G$ form a parallel class in $\mathcal{M}(G,\pi_2\circ\phi)$.
	Item \ref{item:isLinearClass} follows from \ref{item:isTheLiftWeWant}.
    Proposition~\ref{prop:liftsFromLinearClasses} tells us that \ref{item:isTheLiftWeWant} will follow if we show
    that for every $S\subseteq A(G)$,
    \begin{equation}\label{eq:rankOfLiftForSemidirectProductLemma}
        \rank(M^L(G,\phi)_S) =
        \begin{cases}
            \rank(M(G,\phi)_S) & \textnormal{if } S \textnormal{ is } \mathcal{C}\textnormal{-balanced} \\
            1+\rank(M(G,\phi)_S) & \textnormal{otherwise}.
        \end{cases}
    \end{equation}
    It is enough to show that \eqref{eq:rankOfLiftForSemidirectProductLemma} holds for
    every circuit of $\mathcal{M}(M(G,\phi))$.    
    Let $S$ be a circuit of $\mathcal{M}(M(G,\phi))$.
    Up to scaling, there exists a unique linear form $f_S:\mathbb{C}^{A(G)}\rightarrow \mathbb{C}$,
    supported on $S$ such that $f_S(z) = 0$ whenever $z = M(G,\phi)x$ for some $x \in \mathbb{C}^{V(G)}$.
    Let $z^*$ be the column of $M^L(G,\phi)$ indexed by $*$.
    Our task now is to show that $S \in \mathcal{C}$ if and only if $f_S(z^*) = 0$.
    By the second equality in \ref{item:isGainMatroid},
    \cite[Theorem 6.10.5]{oxley2006matroid} implies that $S$ is either
    a balanced cycle of $(G,\pi_2\circ\phi)$,
    or a bicyclic subgraph of $(G,\pi_2\circ\phi)$ with no balanced cycles.

    We now introduce a family of linear forms
    that we will use to write $f_S$ explicitly.
    Let $W = e_1,\dots,e_k$ be a walk in $G$.
    Define the linear form $g_W(z) := \sum_{i=1}^k a_i z_{e_i}$ by
    \[
        a_i := \begin{cases}
            t(\pi_2(\phi(e_1)\phi(e_2)\cdots \phi(e_{i-1}))) & \textnormal{if } e_i \textnormal{ traversed according to its orientation} \\
            -t(\pi_2(\phi(e_i)^{-1}\phi(e_1)\phi(e_2)\cdots \phi(e_{i-1}))) & \textnormal{otherwise}
        \end{cases}        
    \]
    Then $g_W(z)$ is supported on a subset of $W$ and
    when $z = M(G,\pi_2\circ\phi)x$, Remark~\ref{remark:switchingForNewConstruction} implies
    \begin{equation}\label{eq:pathLinearForm}
        g_W(z) = x_{\source(e_1)}-t(\pi_2(\phi(e_1)\cdots\phi(e_k)))x_{\target(e_k)}.
    \end{equation}
    In particular, if $W = S$ is a balanced cycle of $(G,\pi_2\circ\phi)$, then \eqref{eq:pathLinearForm}
    is supported on $S$ and evaluates to zero.
    Therefore $f_S = g_S$.
    In this case, it is a straightforward computation that $f_S(z^*) = 0$ if and only if $S$ is also balanced in $(G,\phi)$,
    i.e.~if $S \in \mathcal{C}$.

    Now assume that $S$ is a bicyclic subgraph of $G$ with no induced cycle that is balanced in $(G,\pi_2\circ \phi)$.
    Let $(W_1,W_2)$ be a covering pair of walks of $S$.
    We claim that
    \begin{equation}\label{eq:linearFormForBicylicSubgraph}
        f_S = (1-t(\pi_2(\phi(W_2))))g_{W_1} - (1-t(\pi_2(\phi(W_1))))g_{W_2}.
    \end{equation}
    It follows from \eqref{eq:pathLinearForm} that the above linear form indeed vanishes on all
    points of the form $z = M(G,\phi)x$.
    Moreover it is not identically zero since $1-c(\pi_2(\phi(W_2))),1-c(\pi_2(\phi(W_1)))\neq 0$,
    as no cycles in $S$ are balanced.
    Thus the claim is proven.

    We now need to show that $f_S(z^*) = 0$ if and only if $S$ is Dutch.
    Expanding and rearranging terms in \eqref{eq:linearFormForBicylicSubgraph} gives
    \[
    	f_S(z^*)=g_{W_1}(z^*)+t(\pi_2(\phi(W_1)))g_{W_2}(z^*) - (g_{W_2}(z^*) + t(\pi_2(\phi(W_2)))g_{W_1}(z^*)).
    \]
    If $W$ is a walk in $G$, then Remark~\ref{remark:luckyCoincidence} implies $g_W(z^*) = -c(\pi_1(\phi(W))$.
    From this, and one more application of Remark~\ref{remark:luckyCoincidence},
    for $i\neq j \in \{1,2\}$ we have
    \[g_{W_i}(z^*)+t(\pi_2(\phi(W_i)))g_{W_j}(z^*) = -c(\pi_1(\phi(W_i)\phi(W_j))).\]
    Since $SO(2)$ is Abelian, Lemma~\ref{lemma:dutch} implies that $f_S(z^*) = 0$ if and only if $S$ is Dutch.
\end{proof}


\begin{prop}\label{prop:mainTheoremTerseVersion}
	Let $\mathcal{S}$ be a subgroup of $\mathbb{R}^2\rtimes SO(2)$
	and let $M = \mathcal{M}(K_n(\mathcal{S}),\pi_2\circ\psi_n)$.
	Let $\mathcal{C}$ be the linear class of $M$ consisting of all
	cycles that are balanced in $(K_n(\mathcal{S}),\psi_n)$,
	and all bicyclic subgraphs of $K_n(\mathcal{S})$ that are Dutch in $(K_n(\mathcal{S}),\psi_n)$.
	Then $\mathcal{M}(\cm_n^\mathcal{S}) = \mathcal{M}(f)$ where $f:2^{A(G)}\rightarrow \mathbb{Z}$
	is defined by
	\[
		f(F) = \begin{cases}
		    2r_M(F) - 1 & \textnormal{if } F \textnormal{ is } \mathcal{C}\textnormal{-balanced} \\
		    2r_M(F) & \textnormal{otherwise}.
		\end{cases}
	\]
\end{prop}
\begin{proof}
    We will denote the conjugate of a complex number or matrix $z$ by $\overline{z}$.
    If we apply the following change of variables (letting $i$ denote $\sqrt{-1}$)
    \[
        x_u \mapsto \frac{x_u + y_u}{2} \qquad y_u \mapsto \frac{x_u - y_u}{2i},
    \]
    then the parameterization $D_n^\mathcal{S}$ of $\cm_n^\mathcal{S}$ becomes
    \[
        d_e = (x_{\source(e)} - t\pi_2\psi_n(e) x_{\target(e)} - c\pi_1\psi_n(e))(y_{\source(e)} -
        \overline{t\pi_2\psi_n(e)} y_{\target(e)} - \overline{c\pi_1\psi_n(e)}).
    \]
    From this, it follows that
    $\cm_n^\mathcal{S}= \aff(M^L(G,\phi))\star \aff(\overline{M^L(G,\phi)})$.
    Since complex conjugation is a field automorphism,
    it does not change the matroid.
    The proposition then follows from Lemma~\ref{lemma:mainLemma} and Theorem~\ref{thm:mainAffine}.
\end{proof}

The last thing we need to prove Theorem~\ref{thm:mainTheoremTranslationsAndRotations}
is the following formula for the rank function of the matroid of a gain graph.

\begin{lemma}[{\cite[Theorem 2.1(j)]{zaslavsky1991biased}}]\label{lemma:rankOfGainMatroid}
    Let $\mathcal{S}$ be a group and let $(G,\phi)$ be an $\mathcal{S}$-gain graph.
    The rank function of $\mathcal{M}(G,\phi)$ is
    \[
    	r_{\mathcal{M}(G,\phi)}(S) = \sum_{F \in C(S)} V(F) - \beta(F)
    \]
    where $\beta(F) = 0$ if $F$ contains an unbalanced cycle, and $\beta(F) = 1$ otherwise.
\end{lemma}

We are now ready to prove Theorem~\ref{thm:mainTheoremTranslationsAndRotations}.

\begin{proof}[Proof of Theorem~\ref{thm:mainTheoremTranslationsAndRotations}]
    Let $\mathcal{C}$ and $f$ be defined as in Proposition~\ref{prop:mainTheoremTerseVersion}.
    Define $g:2^{A(G)}\rightarrow \mathbb{Z}$ by $g(F) = 2|V(F)|-\alpha(F)$.
    Proposition~\ref{prop:mainTheoremTerseVersion} implies that it suffices to show
    $\mathcal{M}(f) = \mathcal{M}(g)$.
    We begin by assuming that $|C(F)|=1$.
    We will show that in this case $f(F) = g(F)$.
	We will repeatedly invoke Lemma~\ref{lemma:rankOfGainMatroid} without explicitly saying so.
	If $F$ is not $\mathcal{C}$-balanced, then $f(F) = 2r_M(F)$.
	In this case, either every cycle in $F$ is balanced in $(K_n(\mathcal{S}),\pi_2\circ\psi_n)$,
	or $F$ contains a non-Dutch bicyclic subgraph with no balanced cycle.
	In the first case, $r_M(F) = |V(F)|-1$ and $\alpha(F) = 2$.
	In the second case, $r_M(F) = |V(F)|$ and $\alpha(F) = 0$.
	Either way, $f(F) = 2|V(F)|-\alpha(F)$.
	Now assume $F$ is $\mathcal{C}$-balanced.
	Then $f(F) = 2r_M(F) - 1$.
	If every cycle in $F$ is balanced in $(K_n(\mathcal{S}),\pi_2\circ\psi_n)$,
	then $r_M(F) = |V(F)|-1$ and $\alpha(F) = 3$.
	Otherwise, $r_M(F) = |V(F)|$ and $\alpha(F) = 1$.
	Again, $f(F) = 2|V(F)|-\alpha(F)$ in either case.

    We now allow $|C(F)|\ge 2$.
    For any $F \subseteq E$, it is easy to see that $\sum_{F' \in C(F)} f(F') \le f(F)$.
    It is also true that $\sum_{F' \in C(F)} g(F') \le g(F)$, and this can be seen as follows
    \begin{align*}
        \sum_{F' \in C(F)}g(F') &= 2|V(F)| - \sum_{F' \in C(F)}\alpha(F')
        \\&\le 2|V(F)| - \min_{F' \in C(F)}\alpha(F') \\
        &= 2|V(F)|-\alpha(F) \\&= g(F).
    \end{align*}
    So if $F$ is independent in $\mathcal{M}(f)$,
    then for any $F' \subseteq F$,
    \[
        |F'| = \sum_{F'' \in C(F')} |F''| \le \sum_{F'' \in C(F')} f(F'') = \sum_{F'' \in C(F')}g(F'')\le g(F')
    \]
    so $F$ is independent in $\mathcal{M}(g)$.
    If $F$ is dependent in $\mathcal{M}(f)$, then we without loss of generality assume $|F| > f(F)$.
    In this case, some $F' \in C(F)$ satisfies $|F'| > f(F')$, since otherwise
    \[
        |F| = \sum_{F' \in C(F)} |F'| \le \sum_{F' \in C(F)} f(F') \le f(F).
    \]
    But if $F'\in C(F)$, then $f(F') = g(G')$, thus implying that $F$ is dependent in $\mathcal{M}(g)$.
\end{proof}

\section{Conclusion}

The main contribution of this paper was to unify several symmetry-forced rigidity results under a single theorem and proof in a way that generalizes to more symmetry groups. The new groups covered by Theorem~\ref{thm:mainTheoremTranslationsAndRotations} are the frieze group generated by a $180^\circ$ rotation and a one-dimensional translation group, and subgroups of $\mathbb{R}^2\rtimes SO(2)$ that violate the discreetness condition of wallpaper and frieze groups (e.g.~the group generated by a five-fold rotation and a two-dimensional lattice of translations).

The most important technical tool behind Theorem~\ref{thm:mainTheoremTranslationsAndRotations} was Theorem~\ref{thm:mainAffine}, a submodular-function-theoretic formula, proven using tropical geometry, for describing the algebraic matroid of a Hadamard product of two affine spaces in terms of the matroids of each, as well as a lift of each such matroid. This proof technique also gives a concise reason why in the two-dimensional case, it has been historically easier to derive symmetry-forced rigidity results for orientation-preserving groups than for orientation-reversing groups -- in the former case, the relevant varieties are Hadamard products of affine spaces, whereas in the latter they are not. Similarly, the two-dimensional Cayley-Menger variety is a Hadamard product of linear spaces, whereas the three-dimensional one is not, thus giving yet another reason why two-dimensional rigidity is so much easier than three-dimensional rigidity.

We end by listing several directions for future research.

\begin{itemize}
    \item Generalize Theorems~\ref{thm:mainResult} and~\ref{thm:mainAffine} to handle more than two linear/affine spaces. In particular, resolve Conjecture~\ref{conj:theOnlyOne}.
    \item Are there other varieties, whose algebraic matroids are relevant for applications, that happen to be Hadamard products of affine spaces? If so, use (a generalization of) Theorem~\ref{thm:mainAffine} to characterize these algebraic matroids.
    \item Extend these tropical techniques to classify generic symmetry forced infinitesimal rigidity for two-dimensional frameworks with symmetry groups that are orientation reversing. The relevant varieties are no longer Hadamard products, but they are images of linear spaces under quadratic monomial maps. Their tropicalizations are therefore obtainable from Bergman fans by summing certain pairs of coordinates.
    \item Via graph pebbling, Theorem~\ref{thm:lamansTheorem} gives rise to a polynomial-time algorithm for recognizing rigidity in the plane~\cite{jacobs1997algorithm}. Perhaps these ideas can be generalized to extract a polynomial-time algorithm from Theorem~\ref{thm:mainTheoremTranslationsAndRotations} for recognizing symmetry-forced rigidity for orientation-preserving groups.
\end{itemize}

\bibliographystyle{plain}
\footnotesize
\bibliography{DTMU}

\end{document}